\documentclass[12pt]{amsart}

\usepackage{amsmath}
\usepackage{amssymb}
\usepackage{pdfsync}

\newcommand{\C}{{\mathbb C}}       
\newcommand{\R}{{\mathbb R}}       
\newcommand{\Rm}{\mathbb{R}^m}
\newcommand{\N}{{\mathbb N}}       %
\newcommand{\Z}{{\mathbb Z}}       
\newcommand{\DD}{{\mathcal D}}

\newcommand{\OO}{{\mathcal O}}

\newcommand{\ZZ}{{\mathcal Z}}

\newcommand{\CH}{{\mathcal Ch}}

\newcommand{\BB}{{\mathcal B}}
\newcommand{\GG}{{\mathcal G}}

\newcommand{\diam}{{\rm diam}}
\newcommand{\dist}{{\rm dist}}
\newcommand{\ra}{\rightarrow}

\newcommand{\rf}[1]{{(\ref{#1})}}

\newcommand{\supp}{{\operatorname{spt}}}

\newcommand{\ve}{{\varepsilon}}

\newcommand{\meas}{{\measuredangle}}

\newcommand{\sst}{{\rm Stop}}

\newcommand{\tree}{{\rm Tree}}

\newcommand{\rest}{{\lfloor}}

\newcommand{\stm}{\setminus}

\newtheorem{theorem}{Theorem}[section]
\newtheorem{lemma}[theorem]{Lemma}

\newtheorem{propo}[theorem]{Proposition}
\newtheorem{pr}[theorem]{Proposition}

\newtheorem*{lemma*}{Lemma}
\newtheorem{lm}[theorem]{Lemma}
\theoremstyle{definition}
\newtheorem{definition}[theorem]{Definition}

\theoremstyle{remark}
\newtheorem{rem}[theorem]{\bf Remark}

\newtheorem*{rem1}{\bf Remark}
\numberwithin{equation}{section}

\newcommand{\brem}{\begin{rem}}
\newcommand{\erem}{\end{rem}}

\begin{document}

\title{Calder\'on-Zygmund kernels and rectifiability in the plane}

\author{V. Chousionis, J. Mateu, L. Prat, and X. Tolsa}
\thanks{Most of this work had been carried out in the first semester of 2011 while V.C was visiting the Centre de Recerca Matem\`atica in Barcelona and he feels grateful for the hospitality. V.C was supported by the Academy of Finland and the grant MTM2010-15657 (Spain). J.M and L.P are supported by grants 2009SGR-000420 (Generalitat de Catalunya) and
MTM2010-15657 (Spain). X.T is supported by grants 2009SGR-000420 (Generalitat de Catalunya) and MTM2010-16232
(Spain).}

\address{Vasilis Chousionis.  Departament de
Ma\-te\-m\`a\-ti\-ques, Universitat Aut\`onoma de Bar\-ce\-lo\-na, Catalonia}

\email{chousionis@mat.uab.cat}

\address{Joan Mateu.  Departament de
Ma\-te\-m\`a\-ti\-ques, Universitat Aut\`onoma de Bar\-ce\-lo\-na, Catalonia}

\email{mateu@mat.uab.cat}

\address{Laura Prat.  Departament de
Ma\-te\-m\`a\-ti\-ques, Universitat Aut\`onoma de Bar\-ce\-lo\-na, Catalonia}

\email{laurapb@mat.uab.cat}
\address{Xavier Tolsa. Instituci\'{o} Catalana de Recerca
i Estudis Avan\c{c}ats (ICREA) and Departament de
Ma\-te\-m\`a\-ti\-ques, Universitat Aut\`onoma de Bar\-ce\-lo\-na.
08193 Barcelona, Catalonia} \email{xtolsa@mat.uab.cat}

\begin{abstract}
Let $E \subset \C$ be a Borel set with finite length, that is, $0<\mathcal{H}^1 (E)<\infty$. By a theorem of David and L\'eger, the $L^2 (\mathcal{H}^1 \lfloor E)$-boundedness of the singular integral associated to the Cauchy kernel (or even to one of its coordinate parts $x / |z|^2,y / |z|^2,z=(x,y) \in \C$) implies that $E$ is rectifiable. We extend this result to any kernel of the form $x^{2n-1} /|z|^{2n}, z=(x,y) \in \C ,n \in \mathbb{N}$. We thus provide the first non-trivial examples of operators not directly related with the Cauchy transform whose $L^2$-boundedness implies rectifiability. 
\end{abstract}
\maketitle

\section{Introduction}
Let $\mu$ be a positive, continuous, that is without atoms, Radon measure  on the complex plane. The Cauchy transform with respect to $\mu$ of a function $f\in L^1_{loc}(\mu)$ is formally defined by
$$C_\mu(f)(z)=\int \frac{f(w)}{z-w} d\mu (w).$$ 
This integral does not usually exist for $z$ in the support of $\mu$ and to overcome this obstacle the truncated Cauchy integrals
$$C_{\mu,\ve} (f)(z)=\int_{|z-w|>\ve}\frac{f(w)}{z-w} d\mu (w), \quad z\in \C, \ve>0,$$
are considered for functions with compact support in any $L^p(\mu),\ 1\leq p\leq \infty$. The Cauchy transform is said to be bounded in $L^2(\mu)$ if there exists some absolute constant $C$ such that
$$\int |C_{\mu,\ve} (f)|^2 d \mu \leq C \int |f|^2 d \mu$$
for all $f \in L^2(\mu)$ and $\ve>0$.  

 We recall that a set in $\R^m$ 
is called $d$-rectifiable if it is contained, up to an $\mathcal{H}^d$-negligible set, in a countable union  of $d$-dimensional Lipschitz graphs; and a Radon measure $\mu$ is $d$-rectifiable if $\mu \ll \mathcal{H}^d$ and it is concentrated on a $d$-rectifiable set, that is, it vanishes out of a $d$-rectifiable set.

The problem of relating the geometric structure of $\mu$ with the $L^2(\mu)$-boundedness of the Cauchy transform has a long history and it is deeply related to rectifiability and analytic capacity. It was initiated by Calder\'on in 1977 with his celebrated paper \cite{cal}, where he proved that the Cauchy transform is bounded on Lipschitz graphs with small constant. 
In \cite{CMM}, Coifman, McIntosh and Meyer removed the small Lipschitz cosntant assumption. Later on David, in \cite{D1}, proved that the rectifiable curves $\Gamma$ for which the Cauchy transform is bounded in $L^2(\mathcal{H}^1 \lfloor \Gamma)$,  are exactly those which satisfy  the linear growth condition, i.e. $$\mathcal{H}^1(\Gamma\cap B(z,r)) \leq C r,\quad z\in \Gamma, r>0,$$ 
where $\mathcal{H}^1 \lfloor \Gamma$ denotes the restriction of the $1$-dimensional Hausdorff measure $\mathcal{H}^1$ on $\Gamma$ and $B(z,r)$ is the closed ball centered at $z$ with radius $r$. 

In the subsequent years there was intense research activity in the topic and new tools and machinery were introduced and studied extensively. 
From the results of Calder\'on, David, and others, soon it became clear that rectifiability plays
an important role in the understanding of the aforementioned problem. In \cite{pj2} Jones gave an intriguing characterization of rectifiability    using the so-called  $\beta$-numbers, which turned out to be very useful in connection with the Cauchy transform, see e.g. \cite{pj1}. In a series of innovative works, see e.g. \cite{DS1} and \cite{DS2}, David and Semmes developed the theory of uniform rectifiability for the geometric study of singular integrals in $\R^m$ on Ahlfors-David regular (AD-regular, for short) measures, that is, measures $\mu$ satisfying
 $$
\frac{r^d}{C} \leq \mu(B(z,r)) \leq C r^d \text{ for }  z \in \operatorname{spt}\mu
\text{ and } 0<r<{\rm diam}(\operatorname{spt}(\mu)),$$ 
for some fixed constant $C$.
Roughly speaking, David and Semmes
intended to find geometric conditions to characterize the AD-regular measures $\mu$ for which some nice singular integrals are bounded in $L^2(\mu)$. To this end they introduced the novel concept of  uniform rectifiability, which can be understood as a quantitative version of rectifiability.
In the $1$-dimensional case, a measure $\mu$ is called uniformly rectifiable if it is AD-regular (with $d=1$) 
and its support is contained in an AD-regular curve. The definition in the case $d>1$ is more technical and we omit it. 

The singular integrals that David and Semmes considered in their works are defined by
odd kernels $K(x)$, smooth on $\R^m \setminus \{0\}$ and satisfying the
usual conditions $|\nabla_j K(x)| \leq C |x|^{-(d+j)} , j=0,1,...$.
The most notable examples of such kernels are the Cauchy kernel and its higher dimensional analogues, the Riesz kernels $\frac{x}{|x|^{d+1}},\ x \in \Rm \stm \{0\}$. 
David showed in \cite{D1} and \cite{D2} that all such singular integrals are bounded in $L^2(\mu)$ when $\mu$ is $d$-uniformly rectifiable. In the other direction David and Semmes proved that the
$L^2(\mu)$-boundedness of all singular integrals in the class
described above forces the measure $\mu$ to be $d$-uniformly
rectifiable. 
 The fundamental question they posed reads as follows: \\
\textit{Does the $L^2(\mu)$-boundedness of the Riesz transform $\mathcal{R}_d$ associated with the kernel $x/|x|^{d+1}$ imply $d$-uniform rectifiability for $\mu$?}

Given three distinct points $z_1,z_2,z_3 \in \C$ their Menger curvature is 
$$c(z_1,z_2,z_3)=\frac{1}{R(z_1,z_2,z_3)},$$ where $R(z_1,z_2,z_3)$ is the radius of the circle passing through $x,y$ and $z$. By an elementary calculation, found by Melnikov \cite{Me} while studying analytic capacity, the Menger curvature is related to the Cauchy kernel  by the formula $$c(z_1,z_2,z_3)^2= \sum_{s \in S_3} \frac{1}{(z_{s_2}-z_{s_1})\overline{(z_{s_3}-z_{s_1})}},$$
where $S_3$ is the group of permutations of three elements. It follows immediately that the permutations of the Cauchy kernel are always positive. This unexpected discovery of Melnikov turned out to be very influential in the study of analytic capacity and the Cauchy transform. In particular it was a crucial tool in the resolution of Vitushkin's conjecture by David, in \cite{dvit}, and in the proof of the semiadditivity of analytic capacity in \cite{tsa}. 

Furthermore, in \cite{Me}, the notion of curvature of a Borel measure $\mu$ was introduced:
$$c^2(\mu)=\iiint c^2(z_1,z_2,z_3) d\mu(z_1)d\mu(z_2)d\mu(z_3).$$Given $\ve>0$, $c_\ve^2 (\mu)$ is the truncated version of $c^2(\mu)$, i.e.\ the above triple integral over the set
$$\{(z_1,z_2,z_3) \in \C ^3:|z_i-z_j| \geq \ve \, \text{for} \, 1\leq i,j \leq 3, \ i\neq j\}.$$
If $\mu$ is finite a Borel measure with linear growth (that is, $\mu(B(z,r))\leq Cr$ for all $z\in\operatorname{spt}\mu$, $r>0$) the relation between the curvature and the $L^2(\mu)$-norm of the Cauchy transform is evident by the following identity proved by Melnikov and Verdera \cite{MeV}:
\begin{equation}
\label{curvl2}
\|C_{\mu,\ve} (1)\|_{L^2(\mu)}^2=\frac{1}{6} c^2_\ve(\mu)+O(\mu (\C)),
\end{equation}
with $|O(\mu (\C))| \leq C \mu (\C)$. 

In \cite{MMV}, Mattila, Melnikov and Verdera settled the David and Semmes question in the case of the Cauchy transform, relying deeply on the use of curvature. They proved that if $E$ is a $1$-AD regular set in the complex plane, the Cauchy singular integral is bounded in $L^2(\mathcal{H}^1 \lfloor E)$ if and only if $E$ is contained in an AD-regular curve, which in the language of David and Semmes translates as $E$ being $1$-uniform rectifiable. 

Later on this result was pushed even further due to the following deep contribution of David and L\'eger. 

\begin{theorem}
[\cite{Leger}]
\label{orleg}
Let $E$ be a Borel set such that $0<\mathcal{H}^1(E)<\infty$,
\begin{enumerate}
\item if $c^2(\mathcal{H}^1 \lfloor E)<\infty$, then $E$ is rectifiable;
\item if the Cauchy transform is bounded in $L^2(\mathcal{H}^1 \lfloor E)$, then $E$ is rectifiable.
\end{enumerate}
\end{theorem}

Notice that (ii) is an immediate consequence of (i)  and (\ref{curvl2}).

Until now, as it is also evident by the (still open) David-Semmes question, very few things were known beyond the Cauchy kernel. In this paper we want to contribute in this direction by extending Theorem \ref{orleg} to a natural class of Calder\'on-Zygmund kernels in the plane. Our starting point was the fact that, somewhat suprisingly, Theorem \ref{orleg} and the main result in \cite{MMV} remain valid if the Cauchy kernel is replaced by one of its coordinate parts $x/|z|^2$ or $y/|z|^2$ for $z =(x,y) \in \C\stm \{0\}$. The kernels we are going to work with consist of a very natural generalisation of these coordinate kernels.

For $n\in \N$, we denote by $T_n$ the singular integral operator associated with the kernel 
\begin{equation}\label{defk44}
K_n(z)=\frac{x^{2n-1}}{|z|^{2n}}, \quad z =(x,y) \in \C\stm \{0\}.
\end{equation}
 Furthermore, for any three distinct $z_1,z_2,z_3\in\C$,
let
$$p_n(z_1,z_2,z_3) = K_n(z_1-z_2)\,K_n(z_1-z_3) + K_n(z_2-z_1)\,K_n(z_2-z_3) + K_n(z_3-z_1)\,K_n(z_3-z_2).$$
Analogously to the definition of the curvature of measures, for any Borel measure $\mu$ let
$$p_n(\mu)=\iiint p_n(z_1,z_2,z_3) d \mu (z_1)d \mu (z_2) d \mu (z_3).$$

Our main result reads as follows.
\begin{theorem}
\label{mainthm}
Let $E$ be a Borel set such that $0<\mathcal{H}^1(E)<\infty$, 
\begin{enumerate}
\item if $p_n(\mathcal{H}^1 \lfloor E)<\infty$, then the set $E$ is rectifiable;
\item if $T_n$ is bounded in $L^2(\mathcal{H}^1 \lfloor E)$, then the set $E$ is rectifiable.
\end{enumerate}
\end{theorem}

The natural question of fully characterizing the homogeneous Calder\'on-Zygmund operators whose boundedness in $L^2(\mathcal{H}^1\lfloor E)$ forces $E$ to be rectifiable, becomes more sensible in the light of our result. We think that such a characterization consists of a deep problem in the area as even the candidate class of ``good" kernels is far from clear. This is illustrated by a result of Huovinen in \cite{hu}, where he showed that  there exist homogeneous kernels, such as $\frac{xy^2}{|z|^4}, z=(x,y)\in \C$, whose corresponding singular integrals are $L^2$-bounded on some purely unrectifiable sets.  We should also remark that in \cite{hdis}, Huovinen proved that the a.e.\ existence of principal values of operators associated to a class of homogeneous vectorial kernels implies rectifiability. This is the case of the complex kernels $\frac{z^{2n-1}}{|z|^{2n}}$, for $n\geq1$, for instance. However, Huovinen's methods do not work for the kernels we are considering in \eqref{defk44}.

Another result in our paper extends the theorem in \cite{MMV} cited above. It reads as follows.

\begin{theorem}
\label{thm22}
Let $\mu$ be a $1$-dimensional AD-regular measure on $\C$ and, for any $n\geq 1$, consider the operator $T_n$ defined above. Then, the measure $\mu$ is uniformly rectifiable if and only if $T_n$ is bounded in $L^2(\mu)$.
\end{theorem}

The fact that uniform rectifiability implies the $L^2(\mu)$-boundedness of $T_n$ is a direct consequence
of David's results in \cite{D1}. The converse implication can be understood as a quantitative version of the assertion (ii) in Theorem \ref{mainthm}. We will prove this by using a corona type decomposition. This is
a technique that goes back to the work of Carleson in the corona theorem, and which has been adapted to the geometric setting of uniform rectifiability by David and Semmes \cite{DS1}.

The paper is organised as follows. In Section 2 we prove that the permutations $p_n$ are positive and behave similarily to curvature on triangles with comparable side lengths and one side of them far from the vertical. 
Sections 3-7 are devoted to the proof of Theorem \ref{mainthm}.
In Section 3 we reduce it to Proposition \ref{mainprop}, which asserts that when $\mu$ is a measure with linear growth supported on the unit ball with mass bigger than $1$ and appropriately small curvature, then there exists a Lipschitz graph which supports a fixed percentage of $\mu$. In Section 4 we give some preliminaries for the proof of Proposition \ref{mainprop}. In Section 5 we follow David and L\'eger in defining suitable stopping time regions and an initial Lipschitz graph. In Section 6 we prove Proposition \ref{mainprop} in the case where the first good approximating line for $\mu$ is far from the vertical axis. The strategy of the proof stems from \cite{Leger} although in many and crucial points (whenever curvature is involved) we need to deviate and provide new arguments. In Section 7 we settle the case where the first approximating line is close to the vertical axis. In this case the scheme of L\'eger does not work. A fine tuning of the stopping time parameters and a suitable covering argument allows us to use the result from the previous section in order to find countable many appropriate Lipschitz graphs that can be joined. 

The proof of Theorem \ref{thm22} is outlined in Section 8. 
As remarked above, a main tool for the proof is the so called corona type decomposition. We will not give 
all the details because many of 
the arguments are similar to the ones for Theorem \ref{mainthm}, adapted to the (simpler)  AD regular case.

Throughout the paper the letter $C$ stands
for some constant which may change its value at different
occurrences. The notation $A\lesssim B$ means that
there is some fixed constant $C$ such that $A\leq CB$,
with $C$ as above. Also, $A\approx B$ is equivalent to $A\lesssim B\lesssim A$.

\section[Permutations of the kernels $K_n$]{Permutations of the kernels $K_n$: positivity and comparability with curvature}

For the rest of the paper we fix some $n \in \N$ and we denote $K:=K_n,\ T:=T_n$ and $p:=p_n$.

\begin{propo}
\label{posperm}
For any three distinct points $z_1,z_2,z_3 \in \C$,
\begin{enumerate}
\item $p(z_1,z_2,z_3) \geq 0$,
\item $p(z_1,z_2,z_3)$ vanishes if and only if $z_1,z_2,z_3$ are collinear.
\end{enumerate}
\end{propo}
\begin{proof}Since $p(z_1,z_2,z_3)$ is invariant by translations, it is enough to estimate the permutations $p(0,z,w)$ for any two distinct points $z=(x,y),w=(a,b) \in \C \stm \{0\}$. Then,
\begin{equation*}
\begin{split}
p(0,z,w)&=K(z)K(w)+K(z)K(z-w)+K(w)K(w-z)\\
&=\frac{x^{2n-1}a^{2n-1}}{|z|^{2n}|w|^{2n}}+\frac{x^{2n-1}(x-a)^{2n-1}}{|z|^{2n}|z-w|^{2n}}-\frac{a^{2n-1}(x-a)^{2n-1}}{|w|^{2n}|z-w|^{2n}}\\
&=\frac{x^{2n-1}a^{2n-1}|z-w|^{2n}+x^{2n-1}(x-a)^{2n-1}|w|^{2n}-a^{2n-1}(x-a)^{2n-1}|z|^{2n}}{|z|^{2n}|w|^{2n}|z-w|^{2n}}.
\end{split}
\end{equation*}
We denote,
\begin{equation*}
A(z,w)=x^{2n-1}a^{2n-1}|z-w|^{2n}+x^{2n-1}(x-a)^{2n-1}|w|^{2n}-a^{2n-1}(x-a)^{2n-1}|z|^{2n},
\end{equation*}
so that
\begin{equation}
\label{pera}
p(0,z,w)=\frac{A(z,w)}{|z|^{2n}|w|^{2n}|z-w|^{2n}},
\end{equation}
and it suffices to prove that $A(z,w)\geq 0$ for all distinct points $z,w \in \C\stm\{0\}$ and $A(z,w)=0$ if and only if $0,z$ and $w$ are collinear.
Furthermore,
\begin{equation*}
\begin{split}
A(z,w)&=x^{2n-1}a^{2n-1}((x-a)^2+(y-b)^2)^n+x^{2n-1}(x-a)^{2n-1}(a^2+b^2)^{n}\\
&\quad \quad-a^{2n-1}(x-a)^{2n-1}(x^2+y^2)^{n} \\
&=x^{2n-1}a^{2n-1}\left(\sum_{k=0}^{n} \binom{n}{k}(x-a)^{2(n-k)}(y-b)^{2k}\right)\\
&\quad \quad+x^{2n-1}(x-a)^{2n-1}\left(\sum_{k=0}^{n} \binom{n}{k}a^{2(n-k)}b^{2k}\right)\\
&\quad \quad-a^{2n-1}(x-a)^{2n-1}\left(\sum_{k=0}^{n} \binom{n}{k}x^{2(n-k)}y^{2k}\right).
\end{split}
\end{equation*}

After regrouping the terms of the last sum we obtain,
\begin{equation*}
\begin{split}
A(z,w)&=\sum_{k=0}^{n} \binom{n}{k}(x^{2n-1}a^{2n-1}(x-a)^{2(n-k)}(y-b)^{2k}
+x^{2n-1}(x-a)^{2n-1}a^{2(n-k)}b^{2k}\\
&\quad \quad \quad \quad -a^{2n-1}(x-a)^{2n-1}x^{2(n-k)}y^{2k} )\\
&=x^{2n-1}a^{2n-1}(x-a)^{2n}+x^{2n-1}(x-a)^{2n-1}a^{2n}
-a^{2n-1}(x-a)^{2n-1}x^{2n}\\
& +\sum_{k=1}^{n} \binom{n}{k}\Big(x^{2n-1}a^{2n-1}(x-a)^{2(n-k)}(y-b)^{2k}
+x^{2n-1}(x-a)^{2n-1}a^{2(n-k)}b^{2k}\\
&\quad \quad \quad \quad -a^{2n-1}(x-a)^{2n-1}x^{2(n-k)}y^{2k} \Big).
\end{split}
\end{equation*}
Since
\begin{equation*}
\begin{split}
x^{2n-1}a^{2n-1}(x-a)^{2n}+x^{2n-1}(x-a)^{2n-1}a^{2n}-a^{2n-1}(x-a)^{2n-1}x^{2n}\\
=x^{2n-1}a^{2n-1}(x-a)^{2n-1}(x-a+a-x)=0,
\end{split}
\end{equation*}
we get,
\begin{equation*}
\begin{split}
A(z,w)=\sum_{k=1}^{n} &\binom{n}{k}x^{2(n-k)}a^{2(n-k)}(x-a)^{2(n-k)} \\ \times &\Big(x^{2k-1}a^{2k-1}(y-b)^{2k}+x^{2k-1}(x-a)^{2k-1}b^{2k}-a^{2k-1}(x-a)^{2k-1}y^{2k} \Big).
\end{split}
\end{equation*}
For $k \in \N$ and $z\neq w \in \C\stm\{0\},z=(x,y),w=(a,b),$ let
\begin{equation}
\label{fk}
F_k(z,w)=x^{2k-1}a^{2k-1}(y-b)^{2k}+x^{2k-1}(x-a)^{2k-1}b^{2k}-a^{2k-1}(x-a)^{2k-1}y^{2k},
\end{equation}
then
\begin{equation}
\label{afk}
A(z,w)=\sum_{k=1}^{n} \binom{n}{k}x^{2(n-k)}a^{2(n-k)}(x-a)^{2(n-k)}F_k(z,w).
\end{equation}
Thus it suffices to prove that for all $k \in \N$,
$$F_k(z,w) \geq 0 \ \text{whenever} \ z \neq w \in \C \stm{\{0\}}$$
and
$$A(z,w)=0 \ \text{if and only if the points} \ 0,z \ \text{and} \ w \ \text{are collinear}.$$

To this end we consider three cases.

\textit{Case 1}, $a=0$.

In this case,
$$A(z,w)=F_n(z,w)=x^{2(2n-1)}b^{2n} \geq 0$$
and since $b \neq 0$,
$$A(z,w)=0 \ \text{if and only if} \ x=0.$$
That is in this case $A(z,w)$ vanishes only when $z$ and $w$ lie on the imaginary axis.

\textit{Case 2,} $b=0$.

In this case,
\begin{equation*}
\begin{split}
F_k(z,w)&=x^{2k-1}a^{2k-1}y^{2k}-a^{2k-1}y^{2k}(x-a)^{2k-1}\\
&=a^{2k-1}y^{2k}(x^{2k-1}-(x-a)^{2k-1}).
\end{split}
\end{equation*} 
Since $a \neq 0$ and the function $x^{2k-1}$ is increasing, 
$x^{2k-1} > (x-a)^{2k-1}$ whenever $a>0$, and $x^{2k-1}<(x-a)^{2k-1}$ whenever $a<0$. Therefore $F_k(z,w) \geq 0$ and $F_k(z,w)=0$ if and only if $y=0$, that is whenever $z$ and $w$ lie in the real axis.

\textit{Case 3,} $a \neq 0$ and $b \neq 0$.\\
In this case by (\ref{fk}), after factoring,
\begin{equation}
\label{fk1}
\begin{split}
F_k(z,w)=a^{2(2k-1)}b^{2k} \Big(\left(\frac{x}{a}\right)^{2k-1} \left(\frac{y}{b}-1\right)^{2k}&+\left(\frac{x}{a}\right)^{2k-1} \left(\frac{x}{a}-1\right)^{2k-1}\\
&- \left(\frac{x}{a}-1\right)^{2k-1}\left(\frac{y}{b}\right)^{2k}\Big).
\end{split}
\end{equation}
We will make use of the following elementary lemma.
\begin{lm}
\label{aux}
Consider the family of polynomials for $k \in \N$,
$$f^k_t(s)=t^{2k-1}(s-1)^{2k}+(t-1)^{2k-1}t^{2k-1}-(t-1)^{2k-1}s^{2k}$$
where $t \in \R$ is a parameter. Then,
$$f^k_t(s) \geq 0, \ \text{for all} \ s\in \R \ \text{and} \ t\in \R$$
and
$$f^k_t(s)=0\ \text{if and only if} \ t=s.$$

\end{lm}
\begin{proof} Since $f_t^{k}$ is an even degree polynomial with respect to $s$,
$$\lim_{s\rightarrow \pm \infty} f^k_t(s)=\lim_{s \rightarrow \pm \infty} \big(t^{2k-1}-(t-1)^{2k-1}\big )s^{2k}=+\infty$$
because $s^{2k}$ is the highest degree monomial of $f_t^{k}$ and the function $t^{2k-1}$ is increasing. 

Furthermore
$$(f^k_t)'(s)=2k t^{2k-1} (s-1)^{2k-1}-2k (t-1)^{2k-1}s^{2k-1}$$
and $$(f^k_t)'(s)=0$$ is satisfied if and only if
$$t^{2k-1} (s-1)^{2k-1}=(t-1)^{2k-1}s^{2k-1},$$
that is if and only if $s=t$.
Therefore $f^k_t(s) \geq f^k_t(t)$ for all $s \in \R$ and
$$f^k_t(t)=t^{2k-1}(t-1)^{2k}+t^{2k-1}(t-1)^{2k-1}-(t-1)^{2k-1}t^{2k}=0.$$
Hence,
$$f_t^k(s) \geq 0 \ \text {for all} \ t,s \in \R$$
and
$$f_t^k(s)=0 \ \text{if and only if} \ t=s.$$
\end{proof}
Hence after applying Lemma \ref{aux} for $t=\frac{x}{a},\ s=\frac{y}{b}$ to (\ref{fk1}), Lemma \ref{posperm} follows.
\end{proof}

Given two distinct points $z,w\in\C$, we will denote by $L_{z,w}$ the line passing through $z,w$.
Given three pairwise different points $z_1,z_2,z_3$, we denote by $\meas(z_1,z_2,z_3)$ the smallest angle formed
by the lines $L_{z_1,z_2}$ and $L_{z_1,z_3}$.
If $L,L'$ are lines, then $\meas(L,L')$ is the smallest angle between $L$ and $L'$. Also, $\theta_V(L):=\meas(L,V)$ where $V$ is a vertical line. Furthermore for a fixed constant $\tau \geq 1$, set 
\begin{equation}
\label{oot}
\OO_\tau = \bigl\{(z_1,z_2,z_3)\in\C^3:\,|z_i-z_j|\leq \tau\,|z_i-z_k|\,\mbox{ for distinct $1\leq i,j,k\leq 3$}\bigr \},
\end{equation}
so that all triangles whose vertices form a triple in $\OO_\tau$ have comparable sides. Finally we should also remark that its not hard to see that for three pairwise different points $z_1,z_2,z_3$,
$$c(z_1,z_2,z_3)=\frac{4 \ \text{area}( T_{z_1,z_2,z_3})}{|z_1-z_2||z_1-z_3||z_2-z_3|}$$
where $T_{z_1,z_2,z_3}$ denotes the triangle determined by $z_1,z_2,z_3$.

\begin{lemma}
\label{compperm}
For $(z_1,z_2,z_3)\in \OO_\tau$, if $\theta_V(L_{z_1,z_2})+\theta_V(L_{z_1,z_3})+ \theta_V(L_{z_2,z_3})\geq \alpha_0>0$, then
$$p(z_1,z_2,z_3)\geq c(\alpha_0,\tau)\,c(z_1,z_2,z_3)^2.$$
\end{lemma}
\begin{proof} It suffices to prove the lemma for $(0,z,w) \in \OO_\tau$ for $z=(x,y),w=(a,b),z \neq w \in \C \stm \{0\}$. From the lemma's assumption we infer that at least one of the angles $\theta_V(L_{z,w}),\theta_V(L_{z,0}),\theta_V(L_{w,0})$ is greater or equal than $\alpha_0/3$. Therefore without loss of generality we can assume that there exists a constant $c_1(\alpha_0)$ such that $|z|\leq c_1(\alpha_0)|x|$. Furthermore let $M:=M(\alpha_0,\tau)>10$ be some large positive number that will be determined later. We distinguish three cases.

\textit{Case 1}: $M^{-1} |x| \leq |x-a|\leq M |x|$ and $M^{-1} |a| \leq |x|\leq M |a|$.\\
By (\ref{pera}), (\ref{afk}) and the fact that the functions $F_k$, as in the proof of Lemma \ref{posperm}, are non-negative we deduce,
\begin{equation*}
\begin{split}
p(0,z,w)&\geq n \frac{x^{2n-2}a^{2n-2}(x-a)^{2n-2}}{|z|^{2n}|w|^{2n}|z-w|^{2n}} F_1(z,w)\\
&=n \frac{x^{2n-2}a^{2n-2}(x-a)^{2n-2}}{|z|^{2n}|w|^{2n}|z-w|^{2n}} (xb-ay)^2\\
&=n \frac{x^{2n-2}a^{2n-2}(x-a)^{2n-2}}{|z|^{2n}|w|^{2n}|z-w|^{2n}} |z|^2|w|^2 \sin ^2(z,w)\\
&=n \left(\frac{|x|}{|z|}\right)^{2n-2}\left(\frac{|a|}{|w|}\right)^{2n-2}\left(\frac{|x-a|}{|z-w|}\right)^{2n-2} \frac{\sin ^2(z,w)}{|z-w|^2}.
\end{split}
\end{equation*}
Recalling that
$$|x|\geq c_1(\alpha_0)^{-1}|z|$$
we notice that in this case,
$$|a|\geq M^{-1}|x| \geq M^{-1} c_1(\alpha_0)^{-1}|z| \geq (Mc_1(\alpha_0)\tau)^{-1}|w|, $$
and in the same manner
$$|x-a|\geq (Mc_1(a_0)\tau)^{-1}|z-w|.$$
Therefore since $c(0,z,w)=\dfrac{2\sin\meas(z,0,w)}{|z-w|}$ we obtain that,
$$p(0,z,w) \geq c(\alpha_0,\tau) c^2(0,z,w).$$

\textit{Case 2}:  $|x-a|<M^{-1}|x|$.
In this case,
$$|x-a|< M^{-1}\tau |z-w|$$
and 
$$|a|\geq \frac{|x|}{2} \geq 2^{-1}c_1(\alpha)^{-1}|z| \geq (2c_1(\alpha_0) \tau)^{-1}|w|.$$

By the definition of $p(0,z,w)$,
\begin{equation}
\label{defp2}
p(0,z,w)=\frac{x^{2n-1}a^{2n-1}}{|z|^{2n}|w|^{2n}}+\frac{x^{2n-1}(x-a)^{2n-1}}{|z|^{2n}|z-w|^{2n}}-\frac{a^{2n-1}(x-a)^{2n-1}}{|w|^{2n}|z-w|^{2n}}.
\end{equation}
Notice that,
\begin{equation*}
\begin{split}
\frac{|x|^{2n-1}|x-a|^{2n-1}}{|z|^{2n}|z-w|^{2n}} &\leq \frac{1}{|z||z-w|} \left( \frac{|x-a|}{|z-w|}\right)^{2n-1} \\
&\leq \frac{\tau}{|z|^2} \left( \frac{|x-a|}{|z-w|}\right)^{2n-1} \leq \tau^{2n} M^{-(2n-1)}|z|^{-2},
\end{split}
\end{equation*}
and in the same way,
$$\frac{|a|^{2n-1}|x-a|^{2n-1}}{|w|^{2n}|z-w|^{2n}} \leq \tau^{2n} M^{-(2n-1)}|z|^{-2}$$
On the other hand,
\begin{equation*}
\begin{split}
\frac{|x|^{2n-1}|a|^{2n-1}}{|z|^{2n}|w|^{2n}}&=\left(\frac{|x|}{|z|}\right)^{2n-1}\left( \frac{|a|}{|w|}\right)^{2n-1} \frac{1}{|z|}\frac{1}{|w|}\\
&\geq (c_1(\alpha_0)^{-2}\tau^{-1}2^{-1})^{2n-1}\tau^{-1}|z|^{-2}
\end{split}
\end{equation*}
Therefore for $M$ large enough and depending only on $\alpha_0$ and $\tau$,
\begin{equation*}
\begin{split}
p(0,z,w)&\geq \big((c_1(\alpha_0)^{-2}\tau^{-1}2^{-1})^{2n-1}\tau^{-1}-2\tau^{2n} M^{-(2n-1)}\big)|z|^{-2} \\
&\geq c(\alpha_0,\tau) c^2(0,z,w).
\end{split}
\end{equation*}

\textit{Case 3}: $|x-a|>M|x|$.

In this case 
$$|x|<M^{-1}|x-a| \leq M^{-1} \tau |z|,$$
and since $M\gg c_1(\alpha_0)+\tau$ we obtain that $|x|<c_1(\alpha_0)^{-1}|z|$ which contradicts the initial assumption, so this case is impossible. 

\textit{Case 4}: $|x|<M^{-1}|a|$.

As with case 3, this case is impossible because it contradicts the initial assumption as
$$|x|<M^{-1} \tau |z|.$$

\textit{Case 5}: $|x|>M |a|$.

In this case we have,
$$|a|<M^{-1}|x|<\frac{|x|}{2}$$
and
$$|x-a|\geq \frac{|x|}{2}\geq (2c_1(\alpha_0))^{-1}|z| \geq (2c_1(\alpha_0) \tau)^{-1}|z-w|.$$
Therefore we can argue as in case 2, recalling (\ref{defp2})
and noticing that now the dominating term is the second one. As in case 2 we deduce that,
\begin{equation*}
\begin{split}
p(0,z,w)&\geq \big((c_1(\alpha_0)^{-2}\tau^{-1}2^{-1})^{2n-1}\tau^{-1}-2\tau^{2n} M^{-(2n-1)}\big)|z|^{-2} \\
&\geq c(\alpha_0,\tau) c^2(0,z,w).
\end{split}
\end{equation*}
\end{proof}

For a positive measure $\mu$ (without atoms, say), denote
$$p(\mu) =\iiint p(z_1,z_2,z_3)\,d\mu(z_1)\,d\mu(z_2)\,d\mu(z_3)$$
and, recalling (\ref{oot}),
$$p_\tau(\mu) =\iiint_{\OO_\tau}
 p(z_1,z_2,z_3)\,d\mu(z_1)\,d\mu(z_2)\,d\mu(z_3).$$
 
We will also use the following notation. Given $z_1\in \C$, we set
$$p_\mu(z_1) = p[z_1,\mu,\mu] = \iint p(z_1,z_2,z_3)\,d\mu(z_2)\,d\mu(z_3),$$
and for $\nu$ another positive measure,
$$p(\nu,\mu,\mu)=\iiint p(z_1,z_2,z_3)\,d\nu (z_1)\,d\mu(z_2)\,d\mu(z_3).$$
For $z_1,z_2\in \C$,
$$p_\mu(z_1,z_2) = p[z_1,z_2,\mu] = \int p(z_1,z_2,z_3)\,d\mu(z_3).$$
We define $p_{\tau,\mu}(z_1)$ and $p_{\tau,\mu}(z_1,z_2)$ analogously.

\section{Reductions}

Our purpose in this section is to reduce the proof of Theorem \ref{mainthm} to the proof of the following proposition which will occupy the biggest part of the paper.

\begin{propo}  
\label{mainprop}
For any constant $C_0\geq 10$, there exists a number $\eta>0$ such that
if $\mu$ is any positive Radon measure on $\C$ satisfying
\begin{itemize}
\item $\mu (B(0,1))\geq 1$,\ $\mu(\C \stm B(0,2))=0$,
\item for any ball $B$, $\mu(B) \leq C_0 \diam (B)$,
\item $p(\mu) \leq \eta$
\end{itemize}
then there exists a Lipschitz graph $\Gamma$ such
that $\mu(\Gamma) \geq 10^{-5}\mu(\C)$. 
\end{propo}

\begin{rem}
\label{invariance}
The previous proposition is equivalent to the following stronger statement. 

For any constant $C_0\geq 10$, there exists a number $\eta>0$ such that
if $\mu$ is any positive Radon measure on $\C$ such that for some bounded Borel $F \subset \C$,
\begin{itemize}
\item $\mu (F)\geq \diam (F)$,
\item for any ball $B$, $\mu(B \cap F) \leq C_0 \diam (B)$
\item $p(\mu \lfloor F) \leq \eta \ \diam(F)$
\end{itemize}
then there exists a Lipschitz graph $\Gamma$ such
that $\mu(\Gamma \cap F) \geq 10^{-5}\mu(F)$. 

Indeed, suppose that Proposition \ref{mainprop} holds. Let $x_0 \in F$ and define the renormalized measure
$$\nu:= \frac{1}{\diam(F)} T_\sharp(\mu \lfloor F),$$ 
where $T(x):=\frac{x-x_0}{\diam(F)}$ and as usual $ T_\sharp(\mu \lfloor F)$ is the image measure of $\mu \lfloor F$ under $T$, defined by $T_\sharp(\mu \lfloor F)(X)=\mu \lfloor F  (T^{-1}(X)), X \subset \C$. Then $\nu(B(0,1)) \geq1, \ \nu(\C \stm B(0,2))=0$ and for any ball $B$, $\nu(B) \leq C_0 \diam (B)$. It also follows easily that for all distinct $x,y,z \in \C$, $p(T(x),T(y),T(z))=\diam(F)^2 p(x,y,z)$, therefore
$$p(\nu)=\frac{\diam(F)^2}{\diam(F)^3} p(\mu \lfloor F) \leq \eta.$$
Hence we can apply Proposition \ref{mainprop} for the measure $\nu $ and obtain a Lipschitz graph $\Gamma$ such that $\nu(\Gamma)\geq 10^{-5} \nu(\C)$, which is equivalent to $\mu(T^{-1}(\Gamma) \cap F) \geq 10^{-5}\mu(F)$ and $T^{-1}(\Gamma)$ is the desired Lipschitz graph.

\end{rem}
We continue with the following lemma which relates $L^2$-boundedness and permutations.

\begin{lemma} 
\label{mmvcl2}
Let $\mu$ be a continuous positive Radon measure in $\C$ with linear growth. If the operator $T$ is bounded in $L^2(\mu)$ then there exists a constant $C$ such that for any ball $B$,
$$\iiint_{B^3}p(x,y,z) d \mu (x)d \mu (y)d \mu (z) \leq C \diam (B).$$
\end{lemma}
For the proof see \cite[Lemma 2.1]{MMV}, where it is stated and proved for the Cauchy transform. The proof goes unchanged if $1/z$ is replaced by any real antisymmetric kernel $k$ with positive permutations satisfying the growth condition $|k(x)| \leq C|x|^{-1}, \ x\in \C\stm\{0\}$. 


For the proof of (i) of Theorem \ref{mainthm} we will need one more lemma.
\begin{lemma}
\label{leg11}
Let $E \subset \C$ be a Borel set with $0<\mathcal{H}^1 (E)<\infty$ and $p(\mathcal{H}^1 \lfloor E)<\infty$. Then for all $\eta>0$ there exists an $F \subset E$ such that,
\begin{enumerate}
\item F is compact,
\item $p(\mathcal{H}^1\lfloor F) \leq \eta \diam F$,
\item $\mathcal{H}^1 (F) >\frac{ \diam F}{40}$,
\item for all $x \in \C, \ t>0$, $\mathcal{H}^1(F \cap B(x,t)) \leq 3t$.
\end{enumerate} 
\end{lemma}
The (fairly easy) proof makes use of standard uniformization arguments and can be found in \cite[Proposition 1.1]{Leger}. Assuming Proposition \ref{mainprop} we can now prove the generalised version of David-L\'eger Theorem.

\begin{proof}[Proof of Theorem \ref{mainthm}]First of all notice that (i) with Lemma \ref{mmvcl2} implies (ii). For the proof of (i) recall that since $\mathcal{H}^1(E)<\infty$, $E$ has a decomposition into a rectifiable and purely unrectifiable part, $E=E_{rect} \cup E_{unrect}$. By way of contradiction assume that $\mathcal{H}^1(E_{unrect})>0$. Now, by Lemma \ref{leg11}, for all $\eta>0$, there exists a compact set $F \subset E_{unrect}$ satisfying
\begin{itemize}
\item $p(\mathcal{H}^1 \lfloor F) \leq \eta \diam F$,
\item $\mathcal{H}^1 (F) >\frac{ \diam F}{40}$,
\item for all $x \in \C, \ t>0$, $\mathcal{H}^1(F \cap B(x,t)) \leq 3t$.
\end{itemize}
Therefore by Remark \ref{invariance}, applied to $F$ and $\mu=40 \mathcal{H}^1 \lfloor F$, there exists a Lipschitz graph $\Gamma$ such that $ \mathcal{H}^1 (\Gamma \cap F) \geq 10^{-5} \mathcal{H}^1 (F)$, which is impossible because $F$ is purely unrectifiable.
\end{proof}

\section{Preliminaries for the proof of proposition \ref{mainprop}}
Let $\mu$ be a positive Radon measure in $\C$. We will say that $\mu$ has $C_0$-linear growth if for all $x \in \C,\, r>0$,
$$\mu(B(x,r))\leq C_0 r.$$
\begin{definition}
For a ball $B=B(x,t)$ we set 
$$\delta_\mu(B)=\delta_\mu (x,t)=\frac{\mu (B(x,t))}{t}.$$ 
\end{definition}

\begin{definition} \label{defbeta}
Given some fixed constant $k>1$, for any ball $B=B(x,t)\subset\C$ and $D$ a line in $\C$, we set
\begin{equation*}
\begin{split}
\beta_{1,\mu}^D(B)&=\frac{1}{t} \int_{B(x,kt)} \frac{\dist(y,D)}t d \mu (y), \\
\beta_{2,\mu}^D (B)&=\left(\frac1t\,\int_{B(x,kt)}\left(\frac{\dist(y,D)}t\right)^2\,
d\mu(y)\right)^{1/2},\\
\beta_{1,\mu}(B)&=\inf_D \beta_{1,\mu}^D (B),\\
\beta_{2,\mu}(B)& =\inf_D \beta_{2,\mu}^D (B)
\end{split}
\end{equation*}
\end{definition}
We will also introduce a small density threshold $\delta >0$ and examine what happens in balls such that $\delta_\mu(B) >\delta$. The following lemma will be used several times.

\begin{lm} [\cite{Leger}, Lemma 2.3.]
\label{balls}
Let $\mu$ be a measure with $C_0$-linear growth. There exist constants $C_1 \geq 1, C_1' \geq 1$ depending only on $C_0$ and $\delta$ such that for any ball $B$ with $\delta_\mu(B) \geq \delta$, there exist three balls $B_1,B_2$ and $B_3$ of radius $\frac{r(B)}{C_1}$ with centers in $B$ such that 
\begin{enumerate}
\item their centers are at least $\frac{12 r(B)}{C_1}$ apart,
\item and $\mu (B_i) \geq \frac{r(B_i)}{C_1'}$ for $i=1,2,3$.
\end{enumerate}
\end{lm}

The following lemma should be considered as a qualitative version of \cite[Lemma 2.5]{Leger}.

\begin{lemma}
\label{q2.5}
Let $\mu$ be a measure with $C_0$-linear growth, and $B\subset\C$ a ball with $\delta_\mu(B)\geq\delta$.
Suppose that $\tau$ is big enough. Then, for any $\ve>0$, there exists some $\delta_1=\delta_1(\delta,\ve)>0$ such that
if 
$$\frac{p_\tau(\mu\rest kB)}{\mu(B)}\leq \delta_1,$$
then $\beta_2(B)\leq \ve.$
\end{lemma}

\begin{proof}
By Lemma \ref{balls} we can find three balls $B_1,B_2,B_3\subset 2B$
with equal radii $C_1^{-1}r(B)$ such that
$\mu(B_i\cap B)\geq \frac{\mu(B)} {C'_1}$ for $i=1,2,3$ and $5B_i\cap 5B_j=\varnothing$ if $i\neq j$.

By Chebyshev, there are sets $Z_i\subset B_i$ with $\mu(Z_i)\approx\mu(B_i)\approx\mu(B)$  such that for $r:=r(B)$ and $z \in Z_i$,
$$p_{\tau,\mu\rest kB}(z)\leq C\,\frac{p_\tau(\mu\rest kB)}r,$$
where here, as well as in the rest of the proof of the lemma, $C$ denotes a constant which depends on $C_1,C'_1,\tau,k,\delta$.
Given $z_1\in Z_1$, we choose $z_2\in Z_2$ such that 
$$p_{\tau,\mu\rest kB}(z_1,z_2)\leq C\,\frac{p_\tau(\mu\rest kB)}{r^2}.$$ 

If $w\in kB\setminus (2B_1\cup 2B_2)$, then $(z_1,z_2,w)\in \OO_\tau$, for $\tau \geq C_1+2k$, and so
either 
$$\theta_V(L_{z_1,z_2})+\theta_V(L_{z_1,w}) + \theta_V(L_{z_2,w})\leq \alpha_0,$$
and so $\dist(w,L_{z_1,z_2})\leq C\,\alpha_0\,r,$ 
or otherwise, by Lemma \ref{compperm},
\begin{equation*}
\begin{split}
p(z_1,z_2,w)&\geq c(\alpha_0,\tau)\, c(z_1,z_2,w)^2 = c(\alpha_0,\tau)\,\frac{\dist(w,L_{z_1,z_2})^2}
{|w-z_1|^2\,|w-z_2|^2} \\
&\geq C(\alpha_0)\,\frac{\dist(w,L_{z_1,z_2})^2}{r^4},
\end{split}
\end{equation*}
with the constants $C(\alpha_0)$ depending on $C_1,C'_1,\tau,k,\delta$ besides $\alpha_0$.
Thus in any case we get
\begin{align*}
\int_{w\in kB\setminus (2B_1\cup2B_2)} &\dist(w,L_{z_1,z_2})^2\,d\mu(w) \\
& \leq 
\int_{w\in kB\setminus (2B_1\cup2B_2)} \bigl[C\,\alpha_0^2 \,r^2+ C(\alpha_0)r^4\,p(z_1,z_2,w)\bigr]\,d\mu(w)\\
& 
\leq C\,\alpha_0^2\,r^3 + C(\alpha_0)\,r^4\,
p_{\tau,\mu\rest kB}(z_1,z_2)\leq C\,\alpha_0^2\,r^3 + C(\alpha_0)\,r^2\,p_\tau(\mu\rest kB).
\end{align*}

Now it remains to see what happens in $2B_1\cup 2B_2$. By Chebyshev,  there exists $z_3\in Z_3$
such that
$$p_{\tau,\mu\rest kB}(z_1,z_3) + p_{\tau,\mu\rest kB}(z_2,z_3)\leq C\,\frac{p_\tau(\mu\rest kB)}{r^2}$$
and
\begin{equation}\label{eqf5}
p(z_1,z_2,z_3)\leq C\,\frac{p_\tau(\mu\rest kB)}{r^3}.
\end{equation}
As above, we deduce that
\begin{equation}\label{eqf6}
\int_{w\in kB\setminus (2B_1\cup2B_3)} \dist(w,L_{z_1,z_3})^2\,d\mu(w)
\leq C\,\alpha_0^2\,r^3 + C(\alpha_0)\,r^2\,p_\tau(\mu\rest kB),
\end{equation}
and also
$$\int_{w\in kB\setminus (2B_2\cup2B_3)} \dist(w,L_{z_2,z_3})^2\,d\mu(w)
\leq C\,\alpha_0^2\,r^3 + C(\alpha_0)\,r^2\,p_\tau(\mu\rest kB).$$

Now we wish to estimate the angle $\meas (L_{z_1,z_2},L_{z_1,z_3})$. 
Recall that
$$c(z_1,z_2,z_3) =\frac{2\sin\meas (L_{z_1,z_2},L_{z_1,z_3})}{|z_2-z_3|},$$
and so
$\sin^2\meas (L_{z_1,z_2},L_{z_1,z_3})\leq C\,c(z_1,z_2,z_3)^2\,r^2.$
Then we deduce that
$$\sin^2\meas (L_{z_1,z_2},L_{z_1,z_3})\leq C\alpha_0^2 + C(\alpha_0)\,p(z_1,z_2,z_3)\,r^2.$$
Notice that, for any $w\in kB$, by elementary geometry we have  
$$\dist(w,L_{z_1,z_2})\leq \dist(w,L_{z_1,z_3}) + C\,r\sin\meas (L_{z_1,z_2},L_{z_1,z_3}).$$
Therefore,
$$\dist(w,L_{z_1,z_2})^2\leq 2\,\dist(w,L_{z_1,z_3})^2 + 
C\alpha_0^2 \,r^2+ C(\alpha_0)\,p(z_1,z_2,z_3)\,r^4.$$
Then, 
from \rf{eqf6} and \rf{eqf5} we obtain
\begin{align*}
& \int_{w\in kB\setminus (2B_1\cup2B_3)} \dist(w,L_{z_1,z_2})^2\,d\mu(w)\\& \leq 
\int_{w\in kB\setminus (2B_1\cup2B_3)}\bigl[
2\,\dist(w,L_{z_1,z_3})^2 + 
C\alpha_0^2 \,r^2+ C(\alpha_0)\,p(z_1,z_2,z_3)\,r^4\bigr]\,d\mu(w)\\
&\leq
C\,\alpha_0\,r^3 + C(\alpha_0)\,p_\tau(\mu\rest kB) \,r^2+ C(\alpha_0)\,p(z_1,z_2,z_3)\,r^5\\
&\leq
C\,\alpha_0\,r^3 + C(\alpha_0)\,p_\tau(\mu\rest kB)\,r^2.
\end{align*}

An analogous argument yields a similar estimate for 
$$\int_{w\in kB\setminus (2B_2\cup2B_3)} \dist(w,L_{z_1,z_2})^2\,d\mu(w).$$
So we get,
$$\int_{w\in kB} \dist(w,L_{z_1,z_2})^2\,d\mu(w)\leq
C\,\alpha_0\,r^3 + C(\alpha_0)\,p_\tau(\mu\rest kB)\,r^2,$$
and thus the lemma follows by taking $\alpha_0$ and 
$p_\tau(\mu\rest kB)/r$ both small enough.
\end{proof}

\section{Construction of a first Lipschitz graph}
As stated above, to construct the Lipschitz graph, we follow quite closely the arguments from
\cite{Leger}. First we need to define a family of stopping time
regions, which are the same as the ones defined in \cite[Subsection 3.1]{Leger}.
Let  $\delta,\ve,\alpha$ be positive constants to be fixed below and choose a point $x_0 \in F$. Then by Lemma \ref{q2.5} there exists a line $D_0$ such that $\beta_1^{D_0} (x_0,1) \leq \ve$. We set
 $$S_{total}= \left\{(x,t)\in F\times (0,5),\begin{array}{ll}
 \mbox{(i)}& \delta(x,t)\geq \frac12\delta\\
 \mbox{(ii)}& \beta_{1}(x,t)<2\ve \\
 \mbox{(iii)}& \exists L_{x,t} \mbox{ s.t. }\left\{\begin{array}{l}
\beta_{1}^{L_{x,t}}(x,t)\leq2\ve,\mbox{ and} \\
\meas(L_{x,t},D_0)\leq\alpha
\end{array}\right.
 \end{array} \!\!\!\!\!\right\}.
 $$
In the definition above to simplify notation we have denoted $\delta(x,r)\equiv\delta_{\mu_{|F}}B(x,r)$ and
$\beta_1(x,r)\equiv\beta_{1,\mu_{|F}}(B(x,r))$. Also 
$L_{x,t}$ stands for some line depending on $x$ and $t$.


For $x\in F$ we set
\begin{equation}\label{defhx}
h(x) = \sup\left\{t>0:\,\exists y\in F, \exists s,\frac
t3\geq s \geq\frac t4, x\in B\bigl(y,\frac s 3\bigr) \mbox{ and }
(y,s)\not\in S_{total}\right\},
\end{equation}
 and
 $$S= \left\{(x,t)\in S_{total}:\,t\geq h(x)\right\}.$$
Notice that if $(x,t)\in S$, then $(x,t')\in S$ for $t'$ such that $t<t'<5$.

Now we consider the following partition of $F$ which depends on the parameters  $\delta,\ve,\alpha$:
\begin{align*}
\ZZ & = \{x\in F:\,h(x)=0\},\\
F_1 & = \left\{x\in F\setminus \ZZ:\, \mbox{$\exists y\in F,
\exists s\in\bigl[ \frac{h(x)}5,\frac{h(x)}2\bigr],x\in
B(y,\frac s 2),\,
\delta(y,s)\leq\delta$}\right\},\\
F_2 & = \Bigl\{x\in F\setminus (\ZZ\cup F_1):\\
&\qquad \left.\mbox{$\exists y\in
F, \exists s \in\!\!\bigl[\frac{h(x)}5,\frac{h(x)}2\bigr],x\in\!
B(y,\frac s 2),\,
\beta_{1}(y, s)\geq\ve$}\right\},\\
F_3 & = \Bigl\{x\in F\setminus (\ZZ\cup F_1\cup F_2):\\
&\qquad \left. \mbox{$\exists y\in F, \exists s \in\bigl[
\frac{h(x)}5,\frac{h(x)}2\bigr],x\in B(y,\frac s 2),\, \meas
(L_{y,t},D_0)\geq\tfrac34\alpha$}\right\}.
\end{align*}
At this point we introduce some thresholds: 
\begin{itemize}
\item $\delta=10^{-10}/N$
\item $\theta_0=\pi /10^6$,
\item $\tau=10^4 C_1$,
\end{itemize}
with $C_1$ appearing first in Lemma \ref{balls} and $N$ is the overlap constant appearing in the Besicovitch covering theorem. Notice that $\tau$ depends on $\delta$, which was fixed earlier, and serves as threshold for the comparability of the triples in $\OO_\tau$. On the other hand $\theta_0$ will be a threshold for the angle $\theta_V (D_0)$. The parameter $\alpha$ will be tuned according to $\theta_V(D_0)$: if $\theta_V(D_0)>\theta_0$ we will choose $\alpha \leq \theta_0 /10$, if $\theta_V(D_0) \leq \theta_0$ then $\alpha=10 \theta_0$. Notice that always $\alpha <\frac{\pi}{10^4}$. Finally we will choose $\ve$ such that $\ve^{\frac{1}{50}}< \alpha$.

In the rest of the section we are going to lay down the necessary background that will lead us to the definition of the Lipschitz graph. We will denote by $\pi$ and $\pi^{\perp}$ the orthogonal projections on $D_0$ and $D_0^{\perp}$ respectively.
\begin{definition}
\label{dandD}
For all $x \in \C$ let 
$$d(x)=\inf_{(X,t) \in S}(d(x,X)+t)$$
and for $p \in D_0$, let
$$D(p)=\inf_{x\in \pi^{-1}(p)}d(x)=\inf_{(X,t)\in S}(d(\pi(X),p)+t).$$
\end{definition}
The following Lemma, whose proof can be found in \cite{Leger}, will be used several times. We state it for the reader's convenience.
\begin{lemma}[\cite{Leger}, Lemma 3.9]
\label{l39}
There exists a constant $C_2$ such that whenever $x,y \in F$ and $t \geq 0$ are such that $d(\pi (x), \pi (y))\leq t$, $d(x)\leq t$, $d(y)\leq t$ then $d(x,y) \leq C_2 t$.
\end{lemma}

We can now define a function $A$ on $\pi (\mathcal{Z})$ by 
$$A(\pi(x))=\pi^{\perp}(x)\ \text{for} \ x \in \mathcal{Z},$$
which is possible because for example by Lemma \ref{l39} $\pi:\mathcal{Z} \ra D_0$ is injective. Furthermore it is not difficult to see that the function $A:\pi(\mathcal{Z}) \ra D_0^{\perp}$ is $2 \alpha$-Lipschitz. In order to extend the function $A$ on the whole line $D_0$ a variant of Whitney's extension theorem is used in \cite{Leger}. Namely after a family of dyadic intervals on $D_0$ is chosen, for any $p \in D_0$ not on the boundaries of the dyadic intervals such that $D(p)>0$ we call $R_p$ the largest dyadic interval containing $p$ such that
$$\diam (R_p) \leq \frac{1}{20} \inf_{u \in R_p} D(u).$$
We relabel the collection of intervals $R_p$ as $\{R_i: i \in I\}$. The $R_i$'s have disjoint interiors and the family $\{2 R_i\}_{i \in I}$ is a covering of $D_0 \stm \pi (\mathcal{Z})$. In the following proposition we gather all their necessary properties for our purposes. For the proof see \cite[Lemma 3.11]{Leger} and the discussion before and afterwards. 

\begin{pr}
\label{ribi} Let $U_0=D_0 \cap B(0,10)$ and $I_0=\{i\in I:R_i \cap U_0 \neq \emptyset\}$. There exists a constant $C_3$ such that
\begin{enumerate}
 \item Whenever $10R_i \cap 10 R_j \neq \emptyset$ then
$$C_3^{-1} \diam (R_j) \leq \diam (R_i) \leq C_3 \diam (R_j).$$
\item For each $i \in I_0$ there exists a ball $B_i \in S$ such that 
$$\diam (R_i) \leq \diam (B_i) \leq C_3 \diam (R_i) \ \text{and} \ d(\pi (B_i),R_i) \leq C_3 \diam (R_i).$$
\end{enumerate}
\end{pr}

Finally let $A_i:D_0 \ra D_0^{\perp}$ be the affine functions with graphs $D_{B_i}$. By the definition of $S_{total}$ the $A_i$'s are $2 \alpha$-Lipschitz. Using an appropriate partition of unity it is not hard to extend $A$ on $U_0 \stm \pi(\ZZ)$ such that it is $C_L \alpha$-Lipschitz on $U_0$, see \cite[p.848-850]{Leger}.

\section{The main step}
For the rest of the section the stopping time regions $\ZZ,F_1,F_2,F_3$, their defining paramaters $\delta,
\ve,\alpha$ and the Lipschitz function $A$ will be as in the previous section. The main step for the proof of Proposition \ref{mainprop} consists in proving the following lemma.
\begin{lemma}
\label{farfromv}
Under the assumptions of Proposition \ref{mainprop}, if furthermore $\theta_V(D_0)> \theta_0$ 
 there exists a Lipschitz graph $\Gamma$ such that $\mu (\Gamma) \geq \frac{99}{100} \mu (\C).$
\end{lemma} 
 
For the proof, we will choose $\Gamma=\{(x,A(x)):x \in U_0\}$ and then we will show that
\begin{equation}
\label{f123}
\mu (F_1)+\mu(F_2)+\mu(F_3) \leq \frac{1}{100} \mu (F)
\end{equation}
because $\ZZ \subset \{(x,A(x)):x \in U_0\}$. 

Clearly Lemma \ref{farfromv} implies Proposition \ref{mainprop} when $\theta_V(D_0) >\theta_0$. In the case $\theta_V(D_0) \leq \theta_0$, which we deal with in Section 7, $\mu (F_1)+\mu(F_2)$ will again be very small. However (\ref{f123}) may fail because $\mu(F_3)$ may be big. In this case the construction of the desired Lipschitz graph, in the sense of  Proposition \ref{mainprop}, will consist of two steps. The first step is similar to the one for the case $\theta_V(D_0) > \theta_0$ although (\ref{f123}) is not guaranteed because as already mentioned $\mu(F_3)$ may be too big. Whenever this happens we can find a family of disjoint balls $\{B_i\}$ which cover a big proportion of $F_3$ and whose best approximating lines are far from the vertical due to the choice of $\alpha= 10 \theta_0$. Then we will apply Lemma \ref{farfromv} to obtain Lipschitz graphs $\Gamma_i$ on each ball $B_i$. The final graph $\Gamma$ will be constructed by connecting the graphs $\Gamma_i$ by line segments. 
 
We start by estimating the measure of $F_2$. Notice that for this lemma we do not need to assume that $\theta_V(D_0)> \theta_0$.
\begin{pr} 
\label{f2}
Under the assumptions of Proposition \ref{mainprop} we have
$$\mu(F_2) \leq 10^{-6}.$$
\end{pr}
\begin{proof} Recalling the definitions of the sets $F_1$ and $F_2$ we deduce that for every $x \in F_2$ there exist $y_x \in F$ and $\tau_x \in [\frac{h(x)}{5}, \frac{h(x)}{2}]$ such that $x\in B(y_x,\tau_x),\beta_1(y_x,\tau_x) \geq \ve$ and $\delta (y_x, \tau_x) >\delta$. Therefore since $k>1$,
$$F_2 \subset \cup_{x \in F_2} B(y_x, k\tau_x).$$
By the 5r-covering Theorem there exists an at most countable set $I$ such that,
\begin{enumerate}
\item $F_2 \subset \cup_{i \in I} B(y_i, 5k\tau_i), y_i \in F$,
\item the balls $B(y_i,k\tau_i)$ are pairwise disjoint,
\item $\beta_1 (y_i, \tau_i) \geq \ve$,
\item $\delta (y_i, \tau_i) >\delta$.
\end{enumerate}
Notice also that since $\mu$ has linear growth,
$$\mu (B(y_i,5k\tau_i)) \leq C_0 5k \tau_i, \ \text{for all} \ i \in I,$$
and by (iv), $\mu (B(y_i, \tau_i))> \delta \tau_i$ hence,
$$\mu (B(y_i,5k \tau_i)) \leq \frac{5kC_0}{\delta} \mu (B(y_i, \tau_i)),i \in I.$$
Furthermore, by Lemma \ref{q2.5}, since $\beta_1 (y_i, \tau_i) \geq \ve$ there exists some $\delta_1=\delta_1(\delta,\ve,\tau)$ such that 
$$\mu (B(y_i, \tau_i)) \leq \frac{p_\tau (\mu \lfloor B(y_i, k\tau_i))}{\delta_1}.$$
Therefore,
\begin{equation*}
\begin{split}
\mu (F_2) &\leq \sum_{i\in I} \mu ( B(y_i, 5k\tau_i)) \leq \frac{5kC_0}{\delta} \sum_{i\in I} \mu (B(y_i,\tau_i))\\
&\leq \frac{5kC_0}{\delta \delta_1} \sum_{i \in I} p_\tau (\mu \lfloor B(y_i,k \tau_i)) \leq \frac{5kC_0}{\delta \delta_1} p_\tau (\mu) \leq \frac{C \eta}{\delta \delta_1}\leq 10^{-6},
\end{split}
\end{equation*}
as $\eta$ will be chosen last and hence much smaller that $\ve, \delta$ and $\delta_1$.
\end{proof}

We know shift our attention to the set $F_1$.
\begin{pr} 
\label{f1}
Under the assumptions of Proposition \ref{mainprop} we have
$$\mu(F_1) \leq 10^{-6}.$$
\end{pr}

\begin{proof} The main point in the proof of this estimate in \cite{Leger} is to show that most of $F$ lies near the graph of $A$. This amounts to showing that 
\begin{equation}
\label{ftildf}
\mu(F\stm \tilde{F}) \leq C \ve^{\frac{1}{2}}
\end{equation} 
as in \cite[Proposition 3.18]{Leger}. The set
\begin{equation}
\label{ftild}
\tilde{F}:=\{x \in F \stm G : d(x,\pi (x)+A(\pi(x))) \leq \ve^{\frac{1}{2}}d(x)\}.
\end{equation}
 can be thought as a good part of $F$ while the definition of $G$ is given in Lemma \ref{314}. Once this is established the desired estimate for $F_1$ is essentialy achieved because as an application of the Besicovich covering theorem it is relatively standard to show that $\mu (F_1 \cap \tilde{F}) \leq 10^{-7}$, see \cite[Proposition 3.19]{Leger}.

The most crucial step for the proof of (\ref{ftildf}) in \cite[Proposition 3.18]{Leger} is \cite[Lemma 3.14]{Leger}. Nevertheless this is the only part in the proof where the curvature is involved, therefore we provide a modified argument in the following lemma. All the other parts in the proof of \cite[Proposition 3.18]{Leger} can be applied to our setting without changes.

\begin{lemma}
\label{314}
For $K>1$, set
$$G=\{x \in F\setminus \ZZ: \exists i, \pi (x) \in 3R_i \ \text{ and } \ x \notin KB_i \} \cup \{x \in F \stm \ZZ: \pi (x) \in \pi (\ZZ)\}.$$
If $K$ is big enough,
$$\mu (G) \leq C \eta,$$
where $C=C(\delta,\theta_0)$.
\end{lemma}
\begin{proof} First suppose that $x \in G\stm \pi^{-1}(\pi(\ZZ))$. Then there exist some $i$ such that $\pi (x) \in 3R_i$ and $x \notin KB_i$. By Proposition \ref{ribi} there exists some absolute constant $C_3>1$ such that if $X_i$ is the center of the ball $B_i$ we have,
$$d(\pi (x), \pi (X_i))\leq C_3 \diam (B_i) \quad \text{and} \quad d(X_i) \leq C_3 \diam(B_i).$$
Let $K>100C_2 C_3$ and $t=\max (d(x),\frac{K}{3C_2}\diam(B_i))$. Then
\begin{equation}
\label{compdk}
d(\pi (x), \pi (X_i))< \frac{K}{3C_2}\diam(B_i) \quad \text{and} \quad d(X_i) < \frac{K}{3C_2}\diam(B_i).
\end{equation}
Then by Lemma \ref{l39} applied to $x,X_i$ and $t$  we deduce that $d(x,X_i)\leq C_2 t$. Now suppose that $d(x) \leq \frac{K}{3C_2}\diam(B_i)$. In this case Lemma {\ref{l39}} would imply that $d(x,X_i) \leq \frac{K}{3C_2}\diam(B_i)$ which is impossible since $x \notin K B_i$. Therefore
\begin{equation}
\label{compdk2}
d(x)>\frac{K}{3C_2}\diam(B_i) \
\text{and} \ d(x,X_i)\leq C_2 d(x).
\end{equation}
Furthermore by the definition of the function $d$ it follows that,
$d(X_i,x)+\diam (B_i)\geq d(x)$ and since $\diam(B_i)< \frac{1}{10} d(x)$ we obtain that $$d(X_i,x)> \frac{9}{10} d(x).$$


Let $\widehat{B_i}$ be a ball centered at $X_i$ such that,
$$B(X_i, \frac{d(x)}{10}) \subset \widehat{B_i}\ \text{and} \ x \in 10\widehat{B_i} \stm 9 \widehat{B_i}.$$
Notice that $\widehat{B_i} \in S$ since it is concentric with $B_i$ and larger. Therefore there exists some line $L_i$ such that
\begin{equation}
\label{bhatins}
\beta_1^{L_i} (\widehat{B_i})<2\ve \ \text{and} \ \meas (L_i,D_0) \leq \alpha.
\end{equation}

By Lemma \ref{balls} there exist two balls $B_1,B_2 \subset 2\widehat{B_i}$ such that,
\begin{enumerate}
\item $r(B_1)=r(B_2)=\frac{r(\widehat{B_i})}{2C_1}$,
\item $d(B_1,B_2) \geq \frac{10 r({\widehat{B_i}})}{C_1}$,
\item $\mu(B_1), \mu(B_2) \geq \frac{r({\widehat{B_i}})}{C'_1}.$
\end{enumerate}
Furthermore let for $j=1,2$,
$$S_j=\{ y \in B_j \cap \widehat{B_i}: d(y,L_i)<10 C'_1 \ve  r(\widehat{B_i})\}.$$
By Chebyshev's inequality and (\ref{bhatins}) it follows that for $j=1,2$,
\begin{equation}
\label{btild}
\mu (S_j)\geq \frac{r(\widehat{B_i})}{2C'_1}.
\end{equation}

Now let $y \in S_1$ and $z \in S_2$. If $\{o\}=L_i \cap L_{y,z}$ then without loss of generality we can assume that $d(o,y) \geq d(y,z)/2$. Therefore,
$$\sin \meas (L_{y,z}, L_i)=\frac{d(y,L_i)}{d(o,y)} \leq \frac{10 C'_1 \ve r(\widehat{B_i})}{5C_1^{-1} r(\widehat{B_i})} \leq 2 C_1 \, C'_1, \ve,$$ 
which combined with (\ref{bhatins}), recalling the fact that $\ve \ll \alpha$, implies that
\begin{equation}
\label{ldo2a}
\meas (L_{y,z},D_0) < 2 \alpha.
\end{equation}
By our choice of $\alpha$ we deduce that $\meas (L_{y,z},D_0) < \frac{ \pi}{100}$.

Let $x^*$ be the orthogonal projection of $x$ on $L_{y,z}$ and consider the following three angles 
$$\theta_1 =\meas (y,x,\pi(x)),\quad \theta_2=\meas (x^*,x,\pi (x)) \quad \text{and} \quad  \theta=\meas(y,x,x^*).$$

Then by elementary geometry we see that $\theta_2=\meas (L_{y,z},D_0)$. Furthermore,
\begin{equation*}
\begin{split}
\sin \theta_1 &=\frac{d(\pi (y),\pi (x))}{d(x,y)} \leq \frac{d(\pi(y),\pi(X_i))+d(\pi(X_i),\pi(x))}{d(X_i,x)-d(X_i,y)} \\ &\leq \frac{r(\widehat{B_i})+C \diam{B_i}}{9r(\widehat{B_i})-r(\widehat{B_i})} \leq \frac{2 r(\widehat{B_i})}{8 r(\widehat{B_i})}\leq \frac{1}{4},
\end{split}
\end{equation*}
because $\diam{B_i} \ll \diam{\widehat{B_i}}$.
Notice that $\theta$ equals either  $\theta_1+\theta_2,\theta_1-\theta_2$ or $\theta_2-\theta_1$, and hence
$$\cos \theta \geq \cos \theta_1 \cos \theta_2 - \sin\theta_1 \sin \theta_2 \geq \frac{1}{8}.$$
So we conclude that for $y \in S_1$ and $z \in S_2$,
\begin{equation}
\label{cr}
\dist (x, L_{yz})=d(x,y) \cos \theta \geq \frac{d(X_i,x)-d(X_i,y)}{8} \geq   r (\widehat{B_i}).
\end{equation}
Notice that for $y \in \widehat{B_1}, z \in \widehat{B_2}$, 
\begin{equation*}
\begin{split}
\frac{10 r(\widehat{B_i})}{C_1} &\leq d(y,z) \leq 2 r(\widehat{B_i}), \\ 
8 r(\widehat{B_i}) &\leq d(x,z) \leq 11r(\widehat{B_i}), \\
8 r(\widehat{B_i}) &\leq d(x,y) \leq 11r(\widehat{B_i}),
\end{split}
\end{equation*}
and thus $(x,y,z) \in \OO_\tau$ for $\tau \geq {11+C_1}$. 

At this point we consider two cases.

\textit{Case 1}: $\theta_V (D_0)> \theta_0$. \\
Recall that in this case  $\alpha \leq \theta_V (D_0)/10$. Hence we conclude that,
$$\theta_V(L_{y,z})> \theta_V (D_0)-\meas (L_{y,z},D_0) > \frac{4\theta_V(D_0)}{5}.$$
Therefore by Lemma \ref{compperm}, taking  $\alpha_0= \theta_0 /10$
$$p(x,y,z) \geq C(\theta_0, \tau )c^2(x,y,z) \ \text{for} \ y \in S_1, z \in S_2$$
and
\begin{equation*}
\begin{split}
\iint_{\{(y,z):(x,y,z) \in \OO_\tau\}} p(x,y,z) d \mu (y) d \mu (z) &\geq \int_{S_1} \int_{S_2} p(x,y,z) d \mu (y) d \mu (z) \\ 
&\geq\int_{S_1} \int_{S_2}  C(\delta,\theta_0)c^2(x,y,z) d\mu (y) d \mu (z).
\end{split}
\end{equation*}
Moreover for $y \in S_1, z \in S_2$, by (\ref{cr}),
$$c^2(x,y,z) = \left( \frac{2\dist(x,L_{y,z})}{d(x,y)d(y,z)}\right)^2 \geq \frac{C}{r(\widehat{B_i})^2},$$
and thus by (\ref{btild}),
$$ \iint_{\{(y,z):(x,y,z) \in \OO_\tau\}} p(x,y,z) d \mu (y) d \mu (z) \geq \frac{C(\delta,\theta_0)}{r(\widehat{B_i})^2} \int_{S_1} \int_{S_2} d\mu (y) d \mu (z) \geq C(\delta,\theta_0).$$

Recaping, we have shown that for all $x \in G \stm \pi^{-1}(\ZZ)$ there exists some constant $C(\delta,\theta_0)$ such that,

\begin{equation}
\label{final314}
\iint_{\{(y,z):(x,y,z) \in \OO_\tau\}} p(x,y,z) d \mu (y) d \mu (z) \geq C(\delta,\theta_0).
\end{equation}

If $x \in G \cap \pi^{-1}(\ZZ)$ we can get the same inequality by repeating the same arguments with the point $X=\pi(x)+A(\pi (x)) \in \ZZ$.

Finally by integrating (\ref{final314}) over all points $x \in G$ deduce that,
\begin{equation*}
\begin{split}
\mu (G) &\leq C(\delta, \theta_0) \iiint_{\OO_\tau} p(x,y,z) d\mu(x) d \mu (y) d \mu (z) \\
&\leq C(\delta, \theta_0) \iiint p(x,y,z) d\mu (x) d \mu (y) d \mu (z) \leq C (\delta,\theta_0) \eta.
\end{split}
\end{equation*}

\textit{Case 2}: $\theta_V (D_0)\leq \theta_0$.

Recalling (\ref{compdk}) and (\ref{compdk2}) we get that $d(\pi (x), \pi (X_i)) \leq \frac{3}{10} r(\widehat{B_i})$. Hence if $x'$ is the projection of $x$ on the line $y+D_0$, where $y \in S_1$, we get that $x' \in 2 \widehat{B_i}$ and $$d (x, y+D_0) \geq \dist(x, 2 \widehat{B_i}) \geq 7 r (\widehat{B_i}).$$ 
Therefore
$$\sin \meas (L_{x,y}, D_0) \geq \frac{7 r (\widehat{B_i})}{11  r (\widehat{B_i})} \geq \frac{7}{11},$$
using that $d(x,y) \leq 11 r(\widehat{B_i})$.
Therefore since $\theta_V (D_0) \leq \theta_0=\frac{\pi}{10^6}$ we deduce that $\meas (L_{x,y}, D_0) >\frac{1}{10}$. Hence we can apply Lemma \ref{compperm} with $\alpha_0=1/10$ and $\tau$ as before to obtain,
$$p(x,y,z) \geq C(\tau)c^2(x,y,z) \ \text{for} \ y \in S_1, z \in S_2.$$
All the other steps of the proof are identical with the previous case.

\end{proof}
\end{proof}


We will now consider the set $F_3$.
\begin{pr}
\label{f3}
Under the assumptions of Lemma \ref{farfromv} we have
$$\mu(F_3) \leq 10^{-6}.$$
\end{pr}
\begin{proof}
Recall that by the assumptions of Lemma \ref{farfromv} $\theta_V(D_0)> \theta_0$ and thus $\alpha \leq \theta_0 / 10$. We start by proving two auxiliary lemmas, the first of them is a substitute of \cite[Lemma 2.5]{Leger}. To simplify notation we let $$p_\lambda (x,t):= \iiint_{\OO_\lambda(x,t)}p(z_1,z_2,z_3)d \mu(z_1) d \mu(z_2) d \mu(z_3)$$
where, recalling (\ref{oot}), $\OO_\lambda(x,t)=\OO_\lambda \cap B(x,t)^3.$
\begin{lm} 
\label{2.5'}
For all $k_0 \geq 1, k \geq 2$ there exists $k_1=k_1(k_0,\delta) \geq 1$ and $C=C(\delta,
\theta_0,k_0) \geq 1$ such that if $B(x,t) \in S$, then for all $y \in B(x, k_0 t)$,
\begin{equation}
\label{b2p}
\beta_1(y,t)^2 \leq C \frac{p_{k_1}(x,t)}{t} \leq C \frac{p_{k_1+k_0}(y,t)}{t}.
\end{equation}
\end{lm}
\begin{proof} Since $B(x,t) \in S$ we have that $\delta (B(x,t)) \geq \delta$. Hence we can apply Lemma \ref{balls}  to find three balls $B_1,B_2,B_3\subset B(x,2t)$
with equal radii $C_1^{-1}t$,
$\mu(B_i\cap B)\geq \frac{\mu(B)} {C'_1}$ for $i=1,2,3$, such that $d(B_i,B_j)\geq \frac{5t}{C_1}$ if $i\neq j$.  Recall that $C_1$ and $C'_1$ depend only on $\delta$. For each ball $B_i, i=1,2,3$, set 
\begin{equation*}
\begin{split}
Z_i=\Big\{u \in F \cap B_i & \cap B(x,t): \\
&\iint_{\{(v,w):(u,v,w)\in \OO_{k_1}(x,t)\}} p(u,v,w) d \mu (v) d \mu (w) \leq C_4 \frac{p_{k_1} (x,t)}{t} \Big\},
\end{split}
\end{equation*}
where $C_4$ is a constant depending on $\delta$ such that by Chebyshev's inequality, $$\mu (Z_i) \geq \frac{t}{2C_1}$$ and $k_1$ will be chosen later. Recall that for $L:=L_{x,t}$, $\beta_1^L (x,t) \leq 2 \ve$ and set for $i=1,2,3$,
\begin{equation}
\label{ziprime}
Z'_i=\{u \in Z_i : d(u, L) \leq C_5 \ve t \},
\end{equation}
where as before $C_5$ is a constant depending on $\delta$ chosen big enough so that by Chebyshev's inequality,
$$\mu (Z'_i) \geq \frac{t}{4 C'_1}.$$ 
For $z_1 \in Z'_1$ applying Chebyshev's inequality once more we can choose $z_2 \in Z'_2$ such that 
\begin{equation}
\label{z1z2w}
\int_{\{w:(z_1,z_2,w) \in \OO_{k_1}(x,t)\}} p(z_1,z_2,w)d \mu (w) \leq C_6 \frac{p_{k_1}(x,t)}{t^2},
\end{equation}
where $C_6$ depends on $\delta$.

From the definition of $Z'_i$ in (\ref{ziprime}), it follows that
$$\meas (L_{z_1,z_2},L)< C(\delta) \ve.$$
Recall also that since $B(x,t) \in S$, $\meas (L,D_0)\leq \alpha$. Then, from the assumptions  $\theta_V(D_0)> \theta_0$ and $\alpha \leq \theta_0 /10$,  since $\ve$ is chosen much smaller than $\alpha$, we deduce that $\theta_V (L_{z_1,z_2}) > \theta_0 / 5$. Therefore by Lemma \ref{compperm} and $\alpha_0=\theta_0 /10$ we obtain for all $w \neq z_1,z_2$,
\begin{equation}
\label{compz1z2}
p(z_1,z_2,w) \geq C(\theta_0,k_1) c^2 (z_1,z_2,w) \geq C(\theta_0,k_1) \left(\frac{\dist(w,L_{z_1,z_2})}{d(w,z_1)d(w,z_2)}\right)^2.
\end{equation}
Furthermore for $w \in F \cap B(x,(k+k_0)t)\stm (2B_1 \cup 2B_2)$ and $i=1,2$,
$$\frac{t}{k_1} \leq \frac{t}{C_1} \leq d(z_i,w)\leq 2(k+k_0)t\leq k_1 t,$$
if $k_1 \geq \max\{2(k+k_0),C_1\}$, hence for such $w$, $(z_1,z_2,w) \in \OO_{k_1}(x,t)$. Therefore, by (\ref{z1z2w}) and (\ref{compz1z2}),
\begin{equation*}
\begin{split}
\int_{B(x,(k+k_0)t)\stm (2B_1 \cup 2B_2)} &\left( \frac{\dist(w,L_{z_1,z_2})}{t}\right)^2 d \mu (w) \\
& \leq C(\delta,\theta_0,k_0) t^2 \int_{\{w:(z_1,z_2,w) \in \OO_{k_1}(x,t)\}}p(z_1,z_2,w) d \mu (w) \\ 
&\leq C(\delta,\theta_0,k_0) p_{k_1}(x,t).
\end{split}
\end{equation*}
Exactly as before, after applying Chebyshev's inequality three times we can find $z_3 \in Z'_3$ such that
\begin{enumerate}
\item $\int_{\{w:(z_1,w,z_3) \in \OO_{k_1}(x,t)\}} p(z_1,w,z_3)d \mu (w) \leq C_7 \dfrac{p_{k_1}(x,t)}{t^2}$,
\item $\int_{\{w:(w,z_2,z_3) \in \OO_{k_1}(x,t)\}} p(w,z_2,z_3)d \mu (w) \leq C_7 \dfrac{p_{k_1}(x,t)}{t^2}$,
\end{enumerate}
and
\begin{equation}
\label{dz3l12}
\left( \frac{d(z_3,L_{z_1,z_2})}{t} \right)^2 \leq C_7\frac{p_{k_1}(x,t)}{t},
\end{equation}
where $C_7$ depends on $\delta$.
Notice that for $w \in 2B_2$, and $i=1,3$,
$$ \frac{t}{k_1}\leq\frac{t}{C_1}\leq d(w,z_i) \leq  k_1 t$$
and $d(z_1,z_3)\geq \frac{t}{k_1}$ as well, therefore
\begin{equation}
\label{2b2xt}
2B_2 \subset \{w: (z_1,w,z_3) \in \OO_{k_1}(x,t)\}.
\end{equation}
Furthermore as before by the definition of the sets $Z'_i$ 
$$\meas (L_{z_1,z_3},L)<C(\delta) \ve.$$
Combined with the assumptions $\meas (L,D_0) \leq \alpha, \theta_V(D_0)>\theta_0 $ and $\alpha \leq \theta_0 /10$ we obtain that for $\ve$ small enough  $\theta_V (L_{z_1,z_3}) > \theta_0 /5$.

Hence  by Lemma \ref{compperm}, for $(z_1,w,z_3) \in \OO_{k_1}(x,t)$ and $\alpha_0=\theta_0/10$,
\begin{equation*}
p(z_1,w,z_3) \geq C(\theta_0, k_1) c^2(z_1,w,z_3),
\end{equation*}
and then, for $w \in B_2$,
$$c^2(z_1,w,z_3)=\left(\frac{2\dist(w,L_{z_1,z_3})}{d(w,z_1)d(w,z_2)} \right)^2 \geq C(\delta) \frac{d(w,L_{z_1,z_2})}{t^4}.$$
Therefore,
\begin{equation}
\begin{split}
\label{b2l13}
\int_{2B_2} &\left( \frac{d(w,L_{z_1,z_3})}{t}\right)^2 d \mu (w) \\
&\leq C(\delta,\theta_0,k_0) \int_{\{w:(z_1,w,z_3)\in \OO_{k_1}(x,t)\}} t^2 p(z_1,w,z_3) d \mu (w) \\
&\leq C(\delta,\theta_0,k_0) p_{k_1} (x,t).
\end{split}
\end{equation}

Let $w'$ be the projection of $w$ on $L_{z_1,z_3}$ and $w''$ the projection of $w'$ on $L_{z_1,z_2}$. Then
\begin{equation*}
d(w,L_{z_1,z_2})^2 \leq d(w,w')^2 \leq 2(d(w,L_{z_1,z_3})^2+d(w',L_{z_1,z_2})^2).
\end{equation*}
By Thales Theorem it follows that $\frac{d(z_1,w')}{d(z_1,z_3)}=\frac{d(w',L_{z_1,z_2})}{d(z_3,L_{z_1,z_2})}$. Hence, since $d(z_1,w')\leq 4(k+k_0)t$ and $d(z_1,z_3)\geq \frac{t}{C_1}$, by (\ref{dz3l12}),
$$\left(\frac{d(w',L_{z_1,z_2})}{t}\right)^2\leq\left(\frac{d(z_3,L_{z_1,z_2})}{t}\right)^2\frac{d(z_1,w')^2}{d(z_1,z_3)^2} \leq C(\delta,k_0)\frac{p_{k_1}(x,t)}{t}.$$
Therefore, using also (\ref{b2l13}),
\begin{equation*}
\begin{split}
\int_{B_2} \left(\frac{d(w,L_{z_1,z_2})}{t}\right)^2 d \mu (w) &\leq 2  \int_{B_2}\left(\left(\frac{d(w,L_{z_1,z_3})}{t}\right)^2+\left(\frac{d(w',L_{z_1,z_2})}{t}\right)^2 \right)d\mu (w) \\
&\leq C(\delta, \theta_0, k_0) p_{k_1}^2(x,t).
\end{split}
\end{equation*}
In the same way we obtain the above estimate for the ball $2B_1$ and therefore,
$$\int_{B(x,(k+k_0)t)}\left(\frac{d(w,L_{z_1,z_2})}{t}\right)^2 d \mu (w) \leq C(\delta, \theta_0,k_0)p_{k_1}(x,t).$$
Since $y \in B(x,(k+k_0)t)$ the above estimate implies that 
$$\beta_2 (y,t)^2 \leq C(\delta, \theta_0,k_0)\frac{p_{k_1}(x,t)}{t}.$$
And the proof of the lemma is complete after observing that for $y \in B(x,k_0t)$, $\OO_{k_1}(x,t) \subset \OO_{k_1+k_0}(y,t)$ implies $p_{k_1}(x,t) \leq p_{k_1+k_0}(y,t).$
\end{proof}
For the second lemma we need to introduce one extra definition.
\begin{definition}
\label{stild}
For $\lambda>1$ we define,
$$\widetilde{S}_{\lambda}=\{(y,t)\in F \times (0,5): B(y,t) \subset B(x_y,s_y) \ \text{where} \ B(x_y,s_y) \in S \ \text{and}\ s_y\leq \lambda t\}.$$
\end{definition}
\begin{lm}
\label{new2.4}
For all $\lambda>1$ there exist constants $k_1 (\delta)$ and $C=C(\delta,\theta_0, \lambda)$ such that
$$\iint_{\widetilde{S}_\lambda} \beta_1(x,t)^2 \frac{d \mu (x) dt}{t} \leq C p_{2\lambda^2(k_1+1)^2} \leq C p(\mu).$$ 
\end{lm}
\begin{proof}  As in \cite[Proposition 2.4]{Leger} for any $\lambda >1$, we obtain 
\begin{equation*}
\begin{split}
\iint_0^\infty & p_{\lambda}(x,t)\frac{d \mu (x) dt}{t^2}\\
&=\iint_0^\infty \left(\iiint \chi_{\OO_{\lambda}(x,t)}(u,v,w)p(u,v,w) d \mu(u) d \mu (v) d \mu (w) \right)\frac{d \mu (x) dt}{t^2}\\
&= \iiint \left(\iint_0^\infty \chi_{\OO_{\lambda}(x,t)}(u,v,w) \frac{d \mu (x) dt}{t^2} \right)p(u,v,w) d \mu (u) d \mu (v) d \mu (w) \\
&\leq C(\lambda) \iiint_{\OO_{2\lambda^2}} p(u,v,w) d \mu (u) d \mu (v) d \mu (w) \\
&\leq C(\lambda) p_{2\lambda^2}(x,t).
\end{split}
\end{equation*}
Since for every $(x,t) \in \widetilde{S}_{\lambda}$  there exists some $(z,s) \in S$ such that $B(x,t) \subset B(z,s)$ and $s \leq \lambda t$, we can apply Lemma \ref{2.5'} with $k_0=1$ in order to get some number $k_1=k_1(\delta)$ and some constant $C =C(\delta,\lambda,\theta_0)$ such that
$$\beta_1(x,\lambda t)^2 \leq C \frac{p_{k_1+1}(x,\lambda t)}{\lambda t}.$$
Notice that for $\lambda \geq 1$, $\OO_{k_1+1}(x,\lambda t) \subset \OO_{\lambda(k_1+1)}(x,t)$ and hence $p_{k_1+1}(x,\lambda t) \leq p_{\lambda (k_1+1)}(x,t)$. Furthermore, $\beta_1(x,t) \leq C(\lambda) \beta_1(x, \lambda t)$, therefore,
$$\beta_1(x,t)^2 \leq C(\delta,\theta_0,\lambda)\frac{p_{\lambda (k_1+1)}(x,t)}{t}.$$
Hence,
\begin{equation*}
\begin{split}
\iint_{\widetilde{S}_{\lambda}} \beta_1(x,t)^2 \frac{d \mu (x) dt}{t}&\leq C \iint_{\tilde{S}_{\lambda}} p_{\lambda(k_1+1)}(x,t) \frac{d \mu (x) dt}{t^2}\\
 &\leq C p_{2\lambda^2(k_1+1)^2} \leq C p (\mu).
\end{split}
\end{equation*}
\end{proof}

In Section 4 of \cite{Leger} one more geometric function is introduced; for $p \in D_0 \cap B(0,10)$ and $t>0$, set
$$\gamma(p,t) = \inf_a \frac{1}{t} \int_{B(p,t) \cap D_0} \frac{|A(u)-a(u)|}{t}du,$$
where the infimum is taken over all affine functions $a: D_0 \ra D_0^{\perp}$. The function $\gamma$ measures how well the the function $A$ can be approximated by affine functions. 

\begin{pr}
\label{41}
Under the assumptions of Lemma \ref{farfromv},
$$\int_{0}^2 \int_{U_0} \gamma(p,t)^2 \frac{dpdt}{t} \leq C(\ve^2+p(\mu)),$$
where $C$ does not depends on $\alpha$.
\end{pr}

Proposition \ref{41} is a substitute of \cite[Proposition 4.1]{Leger} which is one of the key ingredients in the estimate of the measure of $F_3$. Its proof adapts completely to our setting, except one estimate on \cite[p.861]{Leger} where the curvature is involved. In the following we elaborate the argument which bypasses this obstacle. 

Observe that for a given $(p,t) \in F \times (0,5)$ such that $t> \frac{D(p)}{60}$, there exists some $(\widetilde{X},T) \in S$, where $\widetilde{X}:=\widetilde{X}(p,t)$ is such that,
\begin{enumerate}
\item  $d(\pi(\widetilde{X}), p) \leq 60t$,
\item  $T \leq60t$.
\end{enumerate}
Also if $x \in B(\widetilde{X}(p,t),t)$, then $d(\pi (x), p) \leq 61 t$. Let
$$a=\int_{U_0} \int_{\frac{D(p)}{60}}^2 \frac{1}{t} \int_{B(\tilde{X}(p,t),t)} \beta_1 (x,t)^2 d \mu (x) \frac{dt dp}{t}.$$ The following arguments replace the estimate for the term $a$ on \cite[p. 861]{Leger}. We will show that
\begin{equation}
\label{861}
a\leq C p(\mu).
\end{equation}
Notice that for $(p,t,x)$ such that $ p\in U_0, t \in [\frac{D(p)}{60},2]$ and $ x \in B(\widetilde{X}(p,t),t) $ we have that
$x \in B(\widetilde{X}(p,t),60t) \in S$ and $B(x,t) \subset B(\widetilde{X}(p,t),61t)$. Hence recalling Definition \ref{stild}, $(x,t)\in \widetilde{S}_{61}.$

Moreover, as noted earlier, for such triples we also have that $d(\pi(x),p)\leq 61t$.
Therefore, using Fubini and Lemma \ref{new2.4},
\begin{equation*}
\begin{split}
\int_{U_0} \int_{\frac{D(p)}{60}}^2 & \frac{1}{t} \int_{B(\widetilde{X}(p,t),t)} \beta_1 (x,t)^2 d \mu (x) \frac{dt dp}{t}\\
&\leq \iint_{\widetilde{S}_{61}} \beta_1(x,t)^2 \left( \int_{p \in B(\pi(x),61t)}dp\right)\frac{d \mu (x) dt}{t^2} \\
&\leq C \iint_{\widetilde{S}_{61}} \beta_1(x,t)^2 \frac{d \mu (x) dt}{t} \\
&\leq C p(\mu) \leq C \eta.
\end{split} 
\end{equation*}
This finishes the proof of Proposition \ref{41}.

Proposition \ref{41} is used in \cite[Section 5]{Leger} in order to show that the function $A$ cannot oscillate too much. This is the only instance where the curvature is used there, even indirectly. The rest of the arguments in \cite[Section 5]{Leger} which are mainly of Fourier analytic type apply to our setting without any changes. Therefore Proposition \ref{f3} is proven.
\end{proof}
Lemma \ref{farfromv} follows from Propositions \ref{f2}, \ref{f1} and \ref{f3}.

\begin{rem}
\label{invariance2}
Lemma \ref{farfromv} is equivalent to the following more general statement. 

For any constant $C_0\geq 10$, and any $\ve \leq 10^{-300}$ there exists a number $\eta>0$ such that
if $\mu$ is any positive Radon measure on $\C$ such that for some $x_0 \in \C, R>0$,
\begin{itemize}
\item $\beta_1^L (B(x_0,R))\leq \ve$ and $\theta_V(L) > \theta_0$
\item $\mu (B(x_0, R))\geq R$,
\item for any ball $B$, $\mu(B \cap B(x_0,R)) \leq C_0 \diam (B)$
\item $p(\mu \lfloor B(x_0,R)) \leq \eta R$
\end{itemize}
then there exists a Lipschitz graph $\Gamma$ such
that $\mu(\Gamma \cap B(x_0,R)) \geq 10^{-5}\mu(B(x_0,R))$.

The same renormalization argument used in Remark \ref{invariance} works in this case as well. One just needs to notice also, that if 
$$\nu=\frac{1}{R} T_\sharp (\mu \lfloor B(x_0,R))$$ 
for $T=\frac{x-x_0}{R}$, then for the line $L'=\frac{L-x_0}{R}$,
$$\beta^{L'}_{1,\nu} (B(0,1)) \leq \ve,$$
and since $L'$ is parallel to $L$, $\theta_V (L')> \theta_0$.
Furthermore if $\gamma$ is the Lipschitz graph such that $\nu (\gamma) \geq \frac{99}{100} \nu (\C)$ then $\Gamma=T^{-1}(\gamma)$. By the remarks at the end of Section 5 $\gamma$ is a graph of a Lipshitz function $A:L \ra L^{\perp}$ with $\text{Lip}(A)\leq C_L \alpha$ and hence $\Gamma$ is the graph of a Lipschitz function $A':L'\ra {L'}^{\perp}$ with $\text{Lip}(A') \leq C_L \alpha$.
\end{rem}
\section{Proof of Proposition \ref{mainprop}}
The following lemma is the last step needed for the proof of Proposition \ref{mainprop}.
\begin{lm}
\label{vertl} 
Under the assumptions of Proposition \ref{mainprop}, if $\theta_V(D_0)\leq \theta_0$ there exists a Lipschitz graph $\Gamma$ such that 
$$\mu (\Gamma) \geq \frac{99}{10^6} \mu (F).$$
\end{lm}
We start with an auxiliary lemma.
\begin{lm}
\label{closevert}
Suppose that the measure $\mu$ satisfies the assumptions of Proposition \ref{mainprop} and furthermore $\theta_V(D_0) \leq \theta_0$. If we choose $\alpha=10\theta_0 $ and the set $F_3:=F_3(\delta,\ve,\alpha)$ has measure $\mu (F_3) \geq \frac{1}{10} \mu (F)$, then there exist a countable family of balls $B_i$, centered at 
$\widetilde{F}$, recall (\ref{ftild}), and a countable family of lines $L_i$ which satisfy
\begin{enumerate}
\item $\mu(B_i)> \frac{\delta} {2} r(B_i)$
\item $20B_i \cap 20 B_j =\emptyset$ for all $i \neq j$,
\item $\theta_0<\theta_V(L_i) \leq 12 \theta_0$,
\item $\beta_1^{L_i}(B_i) \leq 2 \ve,$
\item $\mu (\cup_i B_i) \geq \frac{1}{3000} \mu (F)$,
\item $p(\mu \lfloor B_i, \mu, \mu)\leq \eta^{\frac{1}{2}} \mu (B_i)$.
\end{enumerate} 
\end{lm}
\begin{proof} For all $x \in F_3 \cap \widetilde{F}$, the balls $B(x,h(x)) \in S$, by the definition of $h(x)$, and furthermore by \cite[Remark 3.3]{Leger}, $\frac{1}{2} \alpha \leq \meas(L_{x,h(x)},D_0) \leq  \alpha$. Therefore since $\theta_V(D_0)\leq \theta_0$ and $\alpha=10 \theta_0$, it follows that 
\begin{equation}
\label{339}
4\theta_0 \leq \theta_V (L_{x,h(x)}) \leq 11 \theta_0.
\end{equation} 

Since $\mu$ has linear growth, for every $1<a<b$ and every ball $B(x,h(x))$ there exists an $(a,b)$-doubling ball $B_x \supset B$ , i.e. $\mu (a B_x) \leq b \mu (B_x)$, whose radius satisfies $\frac{r(B_x)}{r(B)} \leq C(C_0,\delta)$. For our purposes $(100,200)$-doubling balls will be sufficient. Indeed if all balls $B(x,100^j h(x))$ for $0 \leq j \leq m-1$ are not $(100,200)$-doubling then
\begin{equation}
\label{doub}
C_0 100^m h(x)\geq\mu (100^m B(x,h(x))) > 200^m \mu (B(x,h(x)))\geq 200^m \frac{\delta}{2} h(x)
\end{equation}
which is impossible if $m$ is taken big enough. Therefore we can take $$B_x:= B(x,100^mh(x))$$ where $m$ is the smallest integer such that $$\mu (100  B(x,100^m h(x))) \leq 200 \mu (B(x,100^m h(x))).$$ Notice that from (\ref{doub}) we infer that $m <\frac{\log(2C_0 / \delta)}{\log 2}$, hence
$$\frac{r(B_x)}{h(x)}=100^m \leq C(C_0,\delta).$$ 

Furthermore since $B_x \supset B(x,h(x))$, we have that $B_x \in S$. Therefore for the line $L_x:=L_{x,r(B_x)}$  it holds that $\beta_1^{L_x} (B_x) < 2 \ve$. Observe that,
\begin{equation*}
\begin{split}
\beta_1^{L_x} (B(x, h(x)))&= \int_{B(x, kh(x))} \frac{\dist(y, L_x)}{h(x)^2}d \mu (y) \\
&\leq \left(\frac{r(B_x)}{h(x)} \right)^2 \int _{kB_x} \frac{\dist(y, L_x)}{r(B_x)^2}d \mu (y) \leq 2 C(C_0,\delta) \ve.
\end{split}
\end{equation*}
Now we can apply \cite[Lemma 2.6]{Leger} to the ball $B(x,h(x))$, the two lines $L_x , L_{x,h(x)}$ and $\ve_0 =2 C(C_0,\delta) \ve$ in order to obtain that,
$$\meas (L_x,L_{x, h(x)})\leq C(C_0, \delta) \ve.$$
This, combined with (\ref{339}), implies that
$$2 \theta_0\leq\theta_V(L_x)\leq 12 \theta_0.$$
Hence we can apply the $5r$-covering theorem to the family $\{20B_x\}_{x \in F_3 \cap \widehat{F}}$ in order to find a countable family of balls $\{B_i\}_{i\in I} \subset \{B_x\}_{x \in F_3 \cap \widehat{F}}$ and their corresponding lines $\{L_i\}_{i \in I} $ such that,
\begin{itemize}
\item $F_3 \cap \widetilde{F} \subset \cup_{i \in I} 100 B_i,$
\item $20B_i \cap 20B_j= \emptyset, i \neq j,$
\item $\delta (B_i) \geq \frac{1}{2}\delta,$
\item $ \beta_1^{L_i}(B_i) \leq 2 \ve,$
\item $2\theta_0 \leq\theta_V (L_i) \leq 12 \theta_0$.
\end{itemize}
Furthermore 
$$\mu (F_3 \cap \widetilde{F}) \leq   \sum_i \mu(100B_i) \leq \sum_i 200 \mu (B_i) =200 \mu (\cup_i B_i),$$
hence, 
$$\mu(\cup_i B_i) \geq \frac{1}{200} \mu (F_3 \cap \widetilde{F}) \geq \frac{1}{200} \big(\frac{\mu(F)}{10}- \mu (F_3 \setminus \widetilde{F})\big).$$
Recalling (\ref{ftildf}), namely $\mu (F\stm \widetilde{F}) \leq C \ve^{\frac{1}{2}}$, we obtain 
$$\mu(\bigcup_i B_i) \geq \frac{1}{2500} \mu (F).$$
Set,
$$I_G=\{i \in I:p(\mu \lfloor B_i, \mu,\mu)\leq \eta^{1/2} \mu(B_i)\}.$$
Then,

\begin{equation*}
\begin{split}
\mu (\cup_{i \in I \stm I_G} B_i)&= \sum_{i \in I \stm I_G} \mu (B_i) \leq \eta^{-\frac{1}{2}} \sum_{i \in I \stm I_G} p(\mu \lfloor B_i, \mu,\mu) \\
&\leq \eta^{-\frac{1}{2}} p (\mu) \leq \eta^{\frac{1}{2}} \mu (F).
\end{split}
\end{equation*}
By the choice of $\eta$ we conclude that
$$\mu (\cup_{i \in I_G} B_i) \geq \frac{1}{3000} \mu (F).$$
Thus the family $\{B_i\}_{i \in I_G}$ satisfies conditions (i)-(vi) of the Lemma.
\end{proof}

\begin{proof}[Proof of Lemma \ref{vertl}]
By Propositions \ref{f1} and \ref{f2} we have that $\mu (F_1)+ \mu (F_2) \leq 10^{-5}$, if moreover $\mu (F_3) \leq \frac{1}{10}$ we are done. Therefore we can assume that $\mu (F_3) \geq \frac{1}{10}$. We fix $\alpha= \theta_0/10$.

In this case we can  apply Lemma \ref{closevert}  to obtain a ``good" family of balls $B_i=B(x_i,r_i)$ with ``good" approximating lines $L_i$. Then, after choosing $\eta$ small enough, Remark \ref{invariance2} implies the existence of Lipschitz graphs $\Gamma_i$ such that $\mu(B_i \cap \Gamma_i)\geq \frac{99}{100} \mu(B_i)$ and by Lemma   \ref{farfromv},



$$\mu (\cup_{i \in I}( \Gamma_i \cap B_i)) \geq \frac{ 99 }{100}\frac{1}{3000} \mu (F).$$ 
Furthermore, as noted in Remark \ref{invariance2}, the Lipschitz functions $\tilde{A}_i: L_i \ra {L_i}^{\perp}$ ,whose graphs are the $\Gamma_i$'s, are $C_L\theta_0 /10$- Lipschitz, as in this case $\alpha
=\theta_0 /10$. Therefore since $\meas (L_i,D_0)\leq 10 \ \theta_0$ the graphs $\Gamma_i$ are at most $\theta_0(10+\frac{C_L}{5})$-Lipschitz when considered as graphs of functions with domain $D_0$.

Notice that since $\alpha$ is appropriately small, the sets $2B_i \cap \Gamma_i$ are connected. Therefore to conclude the proof it is enough to check that it is possible to join the Lipschitz graphs $\Gamma_i$ with line segments with uniformly bounded slope. 
Recalling that $$\widetilde{F}=\{x \in F \stm G : d(x, \pi(x)+A(\pi (x))) \leq \ve^{\frac{1}{2}} d(x) \},$$
we notice that since the balls $B_i$ are centered in $\widetilde{F}$, they lie very close to the graph of the initial Lipschitz function $A:U_0 \ra D_0^{\perp}$ that was constructed in Section 5.
Therefore, if $i \neq j$, we write $x'_i=\pi(x_i)+A(\pi (x_i))$ and $x'_j=\pi(x_j)+A(\pi (x_j))$ and
$$d(\pi^{\perp} (x'_i),\pi^{\perp}(x'_j)) \leq C_L \alpha d(x'_i, x'_j),$$
as $A$ is $C_L \alpha$-Lipschitz. Furthermore, for $i \neq j$,
\begin{equation*}
\begin{split}
\frac{d(\pi^\perp(x_i),\pi^{\perp}(x_j))}{d(x_i,x_j)} &\leq \frac{d(\pi^{\perp}(x_i),\pi^{\perp}(x'_i))+d(\pi^{\perp}(x'_i),\pi^{\perp}(x'_j))+d(\pi^{\perp}(x'_j),\pi^{\perp}(x_j))}{d(x_i,x_j)}\\
&\leq \frac{\ve^{\frac{1}{2}}(d(x_i)+d(x_j))+C_L\alpha d(x'_i,x'_j)}{d(x_i,x_j)} \\
&\leq \frac{(1+C_L \alpha )\ve^{\frac{1}{2}}(h(x_i)+h(x_j))+C_L \alpha d(x_i,x_j)}{d(x_i,x_j)} \leq C \alpha.
\end{split}
\end{equation*}
The last inequality follows because $d(x_i,x_j) \geq 20(r(B_i)+r(B_j))\geq 20(h(x_i)+h(x_j))$ and $\ve$ is always taken much smaller than $\alpha$. Thus, for all $i \neq j$,
\begin{equation}
\label{joinangle}
\sin \meas (L_{x_i,x_j},D_0) \leq C \alpha.
\end{equation}
Also since $20B_i\cap 20B_j=\emptyset$, for all $y_i \in 2B_i$ and $y_j\in 2B_j$, we have
$$\tan \meas(L_{x_i,x_j},L_{y_i,y_j}) \leq \frac{1}{4},$$
which combined with (\ref{joinangle}) implies that
$$\meas (L_{y_i,y_j},D_0)\leq \frac{\pi}{4}.$$
Therefore the disjoint Lipschitz graphs $2B_i \cap \Gamma_i$ can be joined with line segments with uniformly bounded slope. This completes the proof of Lemma \ref{vertl}.
\end{proof}

Therefore Proposition \ref{mainprop} follows from Lemmas \ref{farfromv} and \ref{vertl}.


\section{Proof of Theorem \ref{thm22}}

In this section we will outline the proof of Theorem \ref{thm22}. Given a $1$-dimensional AD-regular measure $\mu$,
it is already known that
any singular integral with an odd kernel smooth enough is bounded in $L^2(\mu)$ if $\mu$ is uniformly rectifiable. Thus we just have to show that
the $L^2(\mu)$-boundedness of $T$ implies the uniform rectifiability of $\mu$.
As mentioned in the Introduction, we
will not give all the detailed arguments, because they are quite similar 
to the ones for Theorem \ref{mainthm}.

For the proof we need to introduce the ``dyadic cubes'' described in \cite[Chapter 2]{DS1}. These dyadic cubes are not true cubes, but they play this role with respect to a given $1$-dimensional AD regular Borel measure $\mu$, in a sense. To distinguish them from the usual cubes, we will call them {\em $\mu$-cubes}. 

We recall some of the basic properties of the lattice of dyadic $\mu$-cubes.
Given a $1$-dimensional AD regular Borel measure $\mu$ in $\R^d$, for each $j\in\Z$ there exists a family $\DD_j$ of Borel subsets of $\supp\mu$ (the dyadic $\mu$-cubes of the $j$-th generation) such that:
\begin{itemize}
\item[$(a)$] each $\DD_j$ is a partition of $\supp\mu$, i.e.\ $\supp\mu=\bigcup_{Q\in \DD_j} Q$ and $Q\cap Q'=\emptyset$ whenever $Q,Q'\in\DD_j$ and
$Q\neq Q'$;
\item[$(b)$] if $Q\in\DD_j$ and $Q'\in\DD_k$ with $k\leq j$, then either $Q\subset Q'$ or $Q\cap Q'=\emptyset$;
\item[$(c)$] for all $j\in\Z$ and $Q\in\DD_j$, we have $2^{-j}\lesssim\diam(Q)\leq2^{-j}$ and $\mu(Q)\approx 2^{-j}$;
\end{itemize}

We denote $\DD:=\bigcup_{j\in\Z}\DD_j$. For $Q\in \DD_j$, we define the side length
of $Q$ as $\ell(Q)=2^{-j}$. Notice that $\ell(Q)\lesssim\diam(Q)\leq \ell(Q)$.
Actually it may happen that a $\mu$-cube $Q$ belongs to $\DD_ j\cap \DD_k$ with $j\neq k$. In this case, $\ell(Q)$ is not well defined. However, this problem can be solved in many ways.
For example, the reader may think that a $\mu$-cube is not only a subset of $\supp\mu$, but a couple $(Q,j)$, where $Q$ is a subset of $\supp\mu$ and $j\in\Z$ is such that $Q\in\DD_j$. 

Given $a>1$ and $Q\in\DD$, we set
$a Q:= \bigl\{x\in \supp\mu: \dist(x,Q)\leq (a-1)\ell(Q)\bigr\}.$
Also, analogously to the definition of the beta coefficients for balls, we define
\begin{equation*}
\beta_{1}(Q)=\inf_D \frac{1}{\ell(Q)} \int_{3Q} \frac{\dist(y,D)}{\ell(Q)}\, d \mu (y),
\end{equation*}
where the infimum is taken over all the lines $D$. We denote by $L_Q$ a best approximating line for $\beta_1(Q)$.

Following \cite[Chapter 2]{DS1},
 one says that $\mu$ admits a corona decomposition if, for each $\eta>0$ and $\delta>0$, one can find a triple $(\BB,\GG,\tree)$, where $\BB$ and $\GG$ are two subsets of $\DD$ (the ``bad $\mu$-cubes'' and the ``good $\mu$-cubes'') and $\tree$ is a family of subsets $S\subset\GG$, which satisfy the following conditions::
\begin{itemize}
\item[$(a)$] $\DD=\BB\cup\GG$\quad and\quad$\BB\cap\GG=\emptyset.$
\item[$(b)$] $\BB$ satisfies a Carleson packing condition, i.e., 
\begin{equation}\label{eqbcarleson}
\sum_{Q\in\BB:\, Q\subset R}\mu(Q)\lesssim\mu(R)\quad\mbox{ for all $R\in\DD$.}
\end{equation}
\item[$(c)$] $\GG=\bigcup_{S\in\tree}S$ and the union is disjoint.
\item[$(d)$] Each $S\in\tree$ is {\em coherent}. This means that each $S\in\tree$ has a unique maximal element $Q_S$ which contains all other elements of $S$ as subsets, that is $Q'\in S$ as soon as $Q'\in\DD$ satisfies $Q\subset Q'\subset Q_S$ for some $Q\in S$, and that if $Q\in S$ then either all of the children of $Q$ lie in $S$ or none of them do (if $Q\in\DD_j$, the {\em children} of $Q$ is defined as the collection of $\mu$-cubes $Q'\in\DD_{j+1}$ such that $Q'\subset Q$). We say that $S$ is a tree.

\item[$(e)$] The maximal $\mu$-cubes $Q_S$, for $S\in\tree$, satisfy a Carleson packing condition. That is, $\sum_{S\in\tree:\, Q_S\subset R}\mu(Q_S)\lesssim\mu(R)$ for all $R\in\DD$.
\item[$(f)$] For each $S\in\tree$, there exists a (possibly rotated) Lipschitz graph $\Gamma_S$ with constant smaller than $\eta$ such that $\dist(x,\Gamma_S)\leq\delta\,\diam(Q)$ whenever $x\in2Q$ and $Q\in S.$ 
\end{itemize}

It is shown in \cite{DS1} that if $\mu$ is uniformly rectifiable, then it
admits a corona decomposition for all parameters $\eta,\delta>0$. Conversely,
the existence of a corona decomposition for a single set of parameters $\eta,\delta>0$ implies that $\mu$ is uniformly rectifiable. We will show below how one can construct a corona decomposition assuming that $T$ is bounded in $L^2(\mu)$.

Clearly, the $L^2(\mu)$-boundedness of $T$ implies that 
$$p(\mu\rest R)\leq C\,\mu(R)\quad\mbox{for every $R\in\DD$.}$$ 
Then, using Lemma \ref{q2.5}, one easily deduces that, for every $\ve>0$,
\begin{equation}\label{eqBB}
\sum_{\substack{Q\in\DD:\,Q\subset R,\\
\beta_1(Q)\geq \ve}} \mu(Q)\leq C(\ve)\,\mu(R)\quad\mbox{for every $R\in\DD$}.
\end{equation}
In the terminology of \cite{DS1}, this means that $\mu$ satisfies the weak geometric lemma.
Arguing as in \cite[Lemma 7.1]{DS1}, one gets:

\begin{lemma}\label{lem611}
There exists a decomposition $\DD=\BB\cup \GG$ such that \rf{eqbcarleson} holds, and where $\GG$ can be partitioned
into a family $\tree$ of coherent regions $S$ satisfying the following. Setting, for each $S\in\tree$, 
$$\alpha(S)=\frac1{10}\,\theta_0\quad \mbox{if $\,\theta_V(L_{Q_S})>\theta_0 := 10^{-6}\pi$}$$
and
$$\alpha(S)=10\theta_0\quad \mbox{if $\,\theta_V(L_{Q_S})\leq\theta_0,$}$$
we have:
\begin{enumerate}
\item if $Q\in S$, then $\meas(L_Q,L_{Q_S})\leq \alpha(S)$;
\item if $Q$ is a minimal cube of $S$, then at least one of the children of $Q$ lies in $\BB$, or else $\meas(L_Q,L_{Q_S})
\geq \alpha(S)/2$.
\end{enumerate}
\end{lemma}

The lemma is proved by stopping type arguments, using rather standard techniques. As in \cite{DS1} by construction, the set $\BB$ consists of the $\mu$-cubes such that $\beta_1(Q)>\ve$ (for some choice of $\ve\ll\theta_0$), and so it satisfies a Carleson packing condition, as shown above.
The main difference with respect to
\cite[Lemma 7.1]{DS1} is that in the preceding lemma we take two different values for the parameter $\alpha(S)$,
 according to the angle $\theta_V(L_{Q_S})$.

Arguing as in \cite[Proposition 8.2]{DS1}, one gets:

\begin{propo}
For each $S\in\tree$ (from Lemma \ref{lem611}) there exists a Lipschitz function $A_S:L_{Q_S}\to L_{Q_S}^\bot$
with norm $\leq C\,\alpha(S)$ such that, denoting by $\Gamma_S$ the graph of $A_S$,
$$\dist(x,\Gamma_S)\leq C\ve\,\ell(Q)$$
for all $x\in 2Q$, with $Q\in S$.
\end{propo}

To conclude and show that the triple $(\BB,\GG,\tree)$ is a corona decomposition for $\mu$, it remains to prove that 
the maximal $\mu$-cubes $Q_S$, for $S\in\tree$, satisfy a Carleson packing condition. To this end, we need to distinguish
several types of trees. First, we denote by $\sst(S)$ the family of the minimal $\mu$-cubes of $S\in\tree$
(which may be empty). For $Q\in\sst(S)$, 
we write $Q\in\sst_\beta(S)$ if at least one of the children of $Q$ belongs to $\BB$. Also, we set
$Q\in\sst_\alpha(S)$ if $Q\in\sst(S)\setminus \sst_\beta(S)$ and $\meas(L_Q,L_{Q_S})
\geq \alpha(S)/2$. Notice that, by Lemma \ref{lem611}, $\sst(S)= \sst_\alpha(S)\cup \sst_\beta(S)$.
Then we set
\begin{itemize}
\item $S$ is of type $I$ if 
$\mu\left(Q_S \setminus \bigcup_{P\in \sst(S)}P\right) \geq \frac12\,\mu(Q_S).$

\item $S$ is of type $II$ if it is not of type $I$ and
$\mu\left(\bigcup_{\substack{P\in \sst_\beta }}P\right) \geq \frac14\,\mu(Q_S).$

\item $S$ is of type $III$ if it is not of type $I$ or $II$ and
$\mu\left(\bigcup_{\substack{P\in \sst_\alpha} }P\right) \geq \frac14\,\mu(Q_S),$
and moreover $\theta_V(L_{Q_S}) > \theta_0$.

\item $S$ is of type $IV$ if it is not of type $I$, $II$ or $III$, and
$\mu\left(\bigcup_{\substack{P\in \sst_\alpha} }P\right) \geq \frac14\,\mu(Q_S),$
and moreover $\theta_V(L_{Q_S}) \leq \theta_0$.
\end{itemize}
From the definitions above and Lemma \ref{lem611}, it follows easily that any $S\in\tree$ is of type $I$, $II$, $III$, or
$IV$.

To deal with the trees of type $I$, just notice that the sets $Q_S \setminus \bigcup_{P\in \sst(S)}P$, for $S\in\tree$, are pairwise disjoint, and so
\begin{equation}\label{eqI}
\sum_{S\in\tree\cap I:\, Q_S\subset R}\mu(Q_S) \leq 2 \sum_{S\in\tree:\, Q_S\subset R}
\mu\biggl(Q_S \setminus \bigcup_{P\in \sst(S)}P\biggr) \leq 2\,\mu(R).
\end{equation}
If $S$ is a tree of type $II$, then from the definition we infer that
$$\mu(Q_S)\leq C\,\sum_{Q\in\sst(S)}\,\sum_{P\in\BB\cap  \CH(Q)}\mu(P),$$
where the notation $P\in \CH(Q)$ means that $P$ is a child of $Q$. Then it follows that
\begin{equation}\label{eqII}
\sum_{S\in\tree\cap II:\, Q_S\subset R}\mu(Q_S) \leq C\sum_{Q\in\BB:\,Q\subset R}\mu(Q)\leq C\,\mu(R).
\end{equation}
If $S$ is a tree of type $III$ we just sketch the arguments. In this case by combining some of the techniques from the proof of Theorem 
\ref{mainthm} and \cite[Chapters 9-11]{DS1}, 
and denoting 
$$p_Q(\mu) = \iiint_{
\begin{subarray}{l}  x\in3 Q\\ c^{-1}\ell(Q)\leq |x-y|\leq c\,\ell(Q)
\end{subarray}}p(x,y,z)\,d\mu(x)\,d\mu(y)\,d\mu(z)
,$$
for some constant $c$ big enough,
one can show that
\begin{equation}\label{eqd41}
\mu(Q_S)\leq C \sum_{Q\in S} p_Q(\mu).
\end{equation}
So,
\begin{equation}\label{eqIII}
\sum_{S\in\tree\cap III:\, Q_S\subset R} \mu(Q_S)\leq 
C\sum_{Q\subset R}p_Q(\mu)\leq
C\,p(\mu\rest 3R)\leq C\,\mu(R).
\end{equation}
A key point for the proof of \rf{eqd41} is the fact that $\theta_V(L_{Q_S})\geq \theta_0$,
and since $\meas(L_Q,L_{Q_S})\leq \alpha(S) = \theta_0/10$, most of the relevant
triples of points which appear in the estimate of $p_Q(\mu)$ for $Q\in S$ make
triangles with at least one side far from the vertical.

Finally, for a tree $S$ of type $IV$, notice that if $Q\in\sst_\alpha(S)$, then 
$\meas(L_Q,L_{Q_S})\geq \alpha(S)/2 = 5\theta_0$, and thus
$$\theta_V(L_Q)\geq \meas(L_Q,L_{Q_S}) - \theta_V(L_{Q_S}) \geq 4\theta_0.$$
As a consequence, taking into account that $\beta_1(Q)\leq\ve$, assuming $\ve>0$ small enough
one deduces that all the children $P\in\CH(Q)$ satisfy $\theta_V(L_P)\geq 3\theta_0$. Thus these 
$\mu$-cubes $P$ either belong to $\BB$ or are the maximal $\mu$-cubes of some tree of type $I$, $II$, or $III$.
Using also that
$$\mu(Q_S)\leq 4\mu\biggl( \,\bigcup_{Q\in \sst_\alpha(S) }Q\biggr) = 4\sum_{Q\in \sst_\alpha(S) }\mu(Q) = 4
\sum_{Q\in \sst_\alpha(S) }\sum_{P\in\CH(Q)}\mu(P),$$
summing over all the trees $S\in IV$ such that $Q_S\subset R$, one infers that
$$\sum_{S\in\tree\cap IV:\, Q_S\subset R} \mu(Q_S)\leq 4 \sum_{P\in\BB:P\subset R}\mu(P) + 
4
\sum_{\substack{S\in\tree\cap (I\cup II\cup III): \\ Q_S\subset R}} \mu(Q_S)\leq C\,\mu(R),$$
by \rf{eqBB}, \rf{eqI}, \rf{eqII}, and \rf{eqIII}.

Gathering the estimates obtained for the the different types of trees, we get
$$\sum_{S\in\tree:\, Q_S\subset R} \mu(Q_S)\leq C\,\mu(R),$$
as wished. So the triple $(\BB,\GG,\tree)$ is a corona decomposition,  and Theorem \ref{thm22} is proved.

\begin{rem} The following result is due to Mattila, Melnikov and Verdera and is related to \cite{MMV} although it is unpublished. Let $K(z)=|z|^{-1}\Omega (z /|z|), \ z \in \C \stm \{0\},$ where $\Omega$ is an odd function on the unit circle and let $\mu$ be an AD-regular measure. Then if the permutations of $K$ are positive, the $L^2(\mu)$-boundedness of the corresponding  operator $T_{K,\mu}$ implies that $\mu$ is rectifiable. In the following we provide a sketch of the proof. Recall that $\nu$ is a tangent measure of $\mu$ at $z$ if $\nu$ is a locally finite nonzero Borel measure in $\C$ and there exist positive numbers $r_i \ra 0$ such that the measures $r_i^{-1} T_{z,r_i} \sharp \mu$ converge weakly to $\nu$, where $T_{z,r_i}(x)=(x-z) / r_i$. The set of all tangent measures of $\mu$ at $z$ is denoted by $\text{Tan}(\mu,z)$. 
By Lemma \ref{mmvcl2} we obtain that $p_K(\mu)< \infty$ and this implies easily, see \cite{lin}, that for $\mu$ a.e $z \in \C$, $\operatorname{spt} \nu$ is contained in a line for all $\nu \in \text{Tan}(\mu,z)$. Furthermore using standard arguments, as for example in \cite{v}, for $\mu$-a.e. $z \in \C$,
\begin{equation}
\label{tang}
\sup_{0<r<R<\infty}\left|\int_{B(x,R)\stm B(x,r)}K(x-y) d \nu (y)\right| < \infty \ \text{for all} \ x \in \operatorname{spt} \nu .
\end{equation}
Since every $\nu \in \text{Tan}(\mu,z)$ is AD-regular and $\operatorname{spt} \nu$ is contained in a line, (\ref{tang}) implies that $\operatorname{spt} \nu$ is the whole line, see e.g. \cite[Chapter III.1]{DS2}, and hence $\mu$ is rectifiable by \cite[Theorem 16.5]{mb}.
\end{rem}

\begin{rem1} In $\Rm$, $m>2$, results analogous to Proposition \ref{posperm} and Lemma \ref{compperm} hold for the permutations of the kernel $K(x)=\frac{x_1^{2n-1}}{|x|^{2n}},\, x=(x_1,\dots,x_m) \in \Rm \stm \{0\},\,  n\in \N,$ even though by means of different computations than the ones of Section 2. These results and their connection with $1$-rectifiability will appear in a forthcoming paper.
\end{rem1}



\begin{thebibliography}{MTV}

\bibitem[Ca]{cal} A. C. Calder\'on. {\em Cauchy integrals on Lipschitz curves and related operators.} Proc. Nat. Acad. Sci. U.S.A. 74 (1977), no. 4, 1324--1327.

\bibitem[CMM]{CMM} R. Coifman, A. McIntosh and Y. Meyer. {\em L'int\'egrale de Cauchy d\'efinit un op\'erateur born\'e sur $L^{2}$ pour les courbes lipschitziennes.} Ann. of Math. (2) 116 (1982), no. 2, 361--387.

\bibitem[D1]{D1} G. David. {\em Op\'erateurs int\'egraux singuliers sur certaines courbes du plan complexe.} Ann. Sci. \'Ecole Norm. Sup. 17 (1984), no. 1, 157--189.

\bibitem[D2]{D2} G. David.{ \em Op\'erateurs d'int\'egrale singuli\`ere sur les surfaces r\'eguli\`eres.}  Ann. Sci. \'Ecole Norm. Sup. (4) 21 (1988), no. 2, 225--258.

\bibitem[D3]{dvit} G. David. {\em Unrectifiable $1$-sets have vanishing analytic capacity.} Rev. Mat. Iberoamericana 14 (1998), no. 2, 369--479.

\bibitem[DS1]{DS1} G. David and S. Semmes. {\em Singular Integrals and rectifiable sets in $\mathbb {R}
^n$: Au-del\`{a} des graphes lipschitziens.} Ast\'erisque
193, Soci\'{e}t\'{e} Math\'{e}matique de France (1991).

\bibitem[DS2]{DS2} G. David and S. Semmes. {\em Analysis of and on uniformly
rectifiable sets.} Mathematical Surveys and Monographs, 38. American
Mathematical Society, Providence, RI, (1993).

\bibitem[H1]{hdis} P. Huovinen. {\em Singular integrals and rectifiability of measures
in the plane.} Ann. Acad. Sci. Fenn. Math. Diss. 109 (1997).

\bibitem[H2]{hu} P. Huovinen. {\em A nicely behaved singular integral on a purely unrectifiable set.} Proc. Amer. Math. Soc. 129 (2001), no. 11, 3345--3351.

\bibitem[J1]{pj1} P. W. Jones. {\em Square functions, Cauchy integrals, analytic capacity, and harmonic measure. Harmonic analysis and partial differential equations.} (El Escorial, 1987), 24--68, Lecture Notes in Math., 1384, Springer, Berlin, 1989.

\bibitem[J2]{pj2} P. W. Jones. {\em Rectifiable sets and the traveling salesman problem.} Invent. Math. 102 (1990), no. 1, 1--15.

\bibitem[L\'e]{Leger} J. C. L\'eger. {\em Menger curvature and
rectifiability.} Ann. of Math. 149 (1999), 831--869.

\bibitem[Li]{lin} Y. Lin {\em Menger curvature, singular integrals and analytic capacity.} Ann. Acad. Sci. Fenn. Math. Diss. 111 (1997).

\bibitem[M]{mb} P. Mattila {\em Geometry of Sets and Measures in Euclidean Spaces},
Cambridge University Press, (1995).

\bibitem[MMV]{MMV} P. Mattila, M. Melnikov and J. Verdera. {\em The Cauchy integral, analytic capacity, and uniform rectifiability.} Ann. of Math. (2)  144  (1996),  no. 1, 127--136.

\bibitem[M]{Me} M. Melnikov. {\em Analytic capacity: a discrete approach and the curvature of measure.} (Russian) Mat. Sb. 186 (1995), no. 6, 57--76; translation in Sb. Math. 186 (1995), no. 6, 827--846. 

\bibitem[MV]{MeV} M. Melnikov and J. Verdera. {\em A geometric proof of the $L^2$ boundedness of the Cauchy integral on Lipschitz graphs.} Internat. Math. Res. Notices 1995 (7) 325--331.

\bibitem[T]{tsa} X. Tolsa. {\em Painlev\'e's problem and the semiadditivity of analytic capacity.} Acta Math. 190:1 (2003), 105--149.

\bibitem[Vi]{v} M. Vihtil\"a.
{\em The boundedness of Riesz $s$-transforms of measures in $\R^n$.}
Proc. Amer. Math. Soc. 124 (1996), no. 12, 3797--380

\end{thebibliography}
\end{document}